\documentclass[12pt]{amsart}
\usepackage{a4wide}
\usepackage{amssymb, amscd}
\usepackage{amsmath}
\usepackage{amsthm}
\usepackage{caption,color,graphicx}
\usepackage{psfrag}
\usepackage{hyperref}
\usepackage{float}
\usepackage{tikz}
\usepackage{appendix}
\usepackage{color}
\usepackage[all]{xy}
\usepackage{stmaryrd}
\makeindex
\input{inputactonqdef}

\DeclareMathOperator{\alt}{\mbox{Alt}}

\DeclareMathOperator{\id}{id}
\title{Equivariant Algebraic index theorem}
\author{Alexander Gorokhovsky, Niek de Kleijn, Ryszard Nest}
\thanks{Alexander Gorokhovsky was partially supported by an NSF grant. Niek de Kleijn and Ryszard Nest were supported by the Danish National Research Foundation through the Centre for Symmetry and Deformation (DNRF92). Niek de Kleijn was also partially supported by the IAP ``Dygest" of the Belgian Science Policy.}

\begin{document}
 \begin{abstract}
We prove a $\Gamma$-equivariant version of the algebraic index theorem, where $\Gamma$ is a discrete group of automorphisms of a formal deformation  of a symplectic manifold. The  particular cases of this result are  the algebraic version of the transversal index theorem related to the theorem of A. Connes and H. Moscovici for hypoelliptic operators  and the index theorem for the extension of the algebra of pseudodifferential operators by a group of diffeomorphisms of the underlying manifold due to A. Savin, B. Sternin, E. Schrohe and D. Perrot.
 \end{abstract}
\maketitle
\tableofcontents
\section{Introduction}
The term \emph{index theorems} is usually used to describe the equality of, on one hand, analytic invariants of  certain operators on smooth manifolds and, on the other hand, topological/geometric invariants associated to their ``symbols". A convenient way of thinking  about this kind of results is as follows.

\medskip

\noindent One starts with  a $C^*$-algebra of operators  $A$ associated to some geometric situation  and  a $K$-homology cycle  $(A,\pi,H,D)$, where $\pi\colon A\rightarrow B(H)$ is a $*$-representation of $A$ on a Hilbert space $H$ and $D$ is a Fredholm operator on $H$ commuting with the image of $\pi$ modulo compact operators $\mathcal{K}$.   The explicit  choice of the operator $D$ typically has some geometric/analytic flavour, and, depending on the parity of the $K$-homology class, $H$ can have a $\Z/2\Z$ grading such that $\pi$ is  even and $D$ is odd. Given such a (say even) cycle,  an index  of a reduction of $D$ by an idempotent in $A\otimes\mathcal{K}$ defines a pairing of $K$-homology and $K$-theory, i.~e. the group homomorphism
\begin{equation}
\label{eq:pairing}
K_0(\C,A)\times K_0(A,\C)\longrightarrow \Z.
\end{equation}
One can think of this as a Chern character of $D$,
$$
ch(D)\colon K_*(A)\longrightarrow \Z,
$$
and the goal is to compute it explicitly in terms of some topological data extracted from the construction of $D$.
\medskip

\noindent Some examples are as follows.

\medskip

\noindent {\it $A=C(X)$, where $X$ is a compact manifold  and $D$ is an elliptic pseudodifferential  operator acting between spaces of smooth sections of a pair of  vector bundles  on $X$}.

\medskip

\noindent The number $<ch(D),[1]>$ is the Fredholm index of D, i.~e.  the integer
$$
\Ind (D)=\dim(\Ker (D))-\dim(\Coker (D))
$$
and the Atiyah--Singer  index theorem identifies it with the evaluation of the $\hat{A}$-genus of $T^*X$ on the Chern character of the principal symbol of $D$. This is the situation analysed in the original papers of Atiyah and Singer, see \cite{A-S}.

\medskip

\noindent {\it  $A=C^*(\mathcal{F})$, where $\mathcal{F}$ is a foliation of a smooth manifold and $D$  is a transversally elliptic operator on $X$.}

\medskip

\noindent  Suppose that a $K_0(A)$ class is represented by a projection $p\in \mathcal{A}$, where $\mathcal{A}$  is a subalgebra of $A$ closed under holomorphic functional calculus, so that the inclusion $\mathcal{A}\subset A$ induces an isomorphism on K-theory. For appropriately chosen $\mathcal{A}$, the fact that $D$ is transversally elliptic implies that the operator $pDp$ is Fredholm on the range of $p$ and the index theorem identifies the integer $\Ind (pDp)$ with a pairing of a certain cyclic cocycle on $\mathcal{A}$ with the Chern character of $p$ in the cyclic periodic complex of $\mathcal{A}$. For a special class of hypo-elliptic operators see f. ~ex. \cite{C-M}

\medskip

\noindent {\it  Suppose again that $X$ is a smooth manifold. The natural class of representatives of $K$-homology classes of $C(X)$ given by operators of the form $
D=\sum_{\gamma\in \Gamma} P_\gamma \pi(\gamma)
$, where  $\Gamma$ is a discrete group acting on $L^2 (X)$ by Fourier integral operators of order zero and $P_\gamma$ is a collection of pseudodifferential operators on $X$, all of them  of the same (non-negative) order.}

\medskip

\noindent  The principal symbol $\sigma_\Gamma (D)$ of such a $D$ is an element of  the $C^*$-algebra $C(S^*X)\rtimes_{max}\Gamma$, where $S^*M$ is the cosphere bundle of $M$. Invertibility of $\sigma_\Gamma (D)$ implies that $D$ is Fredholm and the index theorem in this case would express  $\Ind_{\Gamma} (D)$ in terms of some equivariant cohomology classes of $M$ and an appropriate equivariant Chern character of  $\sigma_\Gamma (D)$. For the case when $\Gamma$ acts by diffeomorphisms of $M$, see \cite{Sternin, Perrot}.

\medskip

\noindent The typical computation proceeds via a reduction of the  class of operators $D$ under consideration to  an algebra of (complete) symbols, which can be thought of as a "formal deformation"  $\mathcal{A}^\hbar$. Let us spend a few lines on a sketch of the construction of $\mathcal{A}^\hbar$ in the case when the operators in question come from a finite linear combination of diffeomorphisms of a compact manifold  $X$ with coefficients in the algebra $\mathcal{D}_X$ of differential operators on $X$. A special case is of course that of an elliptic differential operator on $X$.

\begin{example}
Let $\Gamma$ be a subgroup of the group of diffeomorphisms of $X$ viewed as a discrete group. $\Gamma$ acts naturally on $\mathcal{D}_X$. Let $\mathcal{D}_X^\sbullet$ be the filtration by degree of $\mathcal{D}_X$. Then the associated Rees algebra
$$
R=\{(a_0,a_1,\ldots )\mid a_k\in \mathcal{D}_X^k \}
$$
with the product
$$
(a_0,a_1,\ldots )(b_0,b_1,\ldots )=(a_0b_0,a_0b_1+a_1b_0,\ldots ,\sum_{i+j=k}a_ib_j,\ldots )
$$
has the induced action of $\Gamma$. The shift
$$
\hbar\colon (a_0,a_1,\ldots )\rightarrow (0, a_0,a_1,\ldots )
$$
makes $R$ into an $\C\hhbar$-module and $R/\hbar R$ is naturally isomorphic to $\prod_k Pol_k(T^*X)$ where $Pol_k(T^*X)$ is the space of  smooth, fiberwise polynomial functions of degree $k$ on the cotangent bundle $T^*X$.  A choice of an isomorphism of $R$ with $\prod_k Pol_k(T^*X)\hhbar$  induces on  $\prod_k Pol_k(T^*X)\hhbar $ an associative, $\hbar$-bilinear product $\star$, easily seen to extend to $C^\infty (T^*X)$. Since  $\Gamma$ acts  by automorphisms on $R$, it also acts on $(C^\infty (T^*X)\hhbar,\star )$.

\end{example}
This is usually formalized in the following definition.

\begin{definition}

 A formal deformation quantization of a symplectic manifold $(M,\omega )$ is an associative $\mathbb{C}\hhbar$-linear  product $\star$ on
$C^\infty(M)\hhbar$ of the form
\[
f\star g=fg+\frac{i\hbar}{2}\{f,g\} +\sum_{k\geq 2} \hbar^k P_k(f,g);
\]
where $\{f,g\}:=\omega(I_\omega(df),I_\omega(dg))$  is the canonical Poisson bracket  induced by the symplectic structure, $I_\omega$ is the isomorphism of $T^*M$ and $TM$ induced by $\omega$, and the $P_k$ denote bidifferential operators. We will also require that $f\star 1=1\star f=f$ for all $f\in C^\infty(M)\hhbar$. We will use $\AH(M)$ to denote the algebra $(C^\infty(M)\hhbar,\star )$. The ideal $\AHc(M)$ in $\AH(M)$, consisting of power series of the form
$\sum_k\hbar^k f_k$, where $f_k$ are compactly supported, has a unique (up to a normalization)
trace $Tr$ with values in ${\mathbb C}[\hbar^{-1}, \hbar\rrbracket$ (see f. ex. \cite{Fb}).

\end{definition}


\medskip

\noindent It is not difficult to see that the index computations (as in \ref{eq:pairing}) reduce to the computation of the pairing of the trace (or some other cyclic cocycle) with the $K$-theory of the symbol algebra, which, in the example above, is identified with a crossed product $\AHc(M)\rtimes \Gamma$. An example of this reduction is given in  \cite{N-T2}.

\medskip

\noindent {\it Since the product in $\AHc(M)$ is local, the computation of the pairing of $K$-theory and cyclic cohomology of $\AHc(M)$ reduces to a differential-geometric problem and the result  is usually called the ``algebraic index theorem}".

 \begin{remark}
 Since cyclic periodic homology is invariant under (pro)nilpotent extensions, the result of the pairing depends only on the $\hbar =0$ part of the $K$-theory of $\AHc(M)\rtimes \Gamma$. In our example, the $\hbar=0$ part of the symbol algebra $\AHc(M)\rtimes \Gamma$ is just $C^\infty_c (M)\rtimes \Gamma$, hence the Chern character of $D$ enters into the final result only through a class in the equivariant cohomology $H_\Gamma^*(M)$.
 \end{remark}

\subsection{The main result.} Suppose that $\Gamma$ is a  discrete group acting by continuous automorphisms on a formal deformation $\AH(M)$ of a symplectic manifold $M$. Let $\AG$ denote the algebraic crossed product associated to the given action of $\Gamma$.
For a non-homogeneous group cocycle $\xi\in C^k(\Gamma, \C)$, the formula below defines a cyclic $k$-cocycle $Tr_\xi$ on $\AHc(M)\rtimes\Gamma$.
\begin{equation}
Tr_\xi(a_0\gamma_0\otimes\ldots\otimes a_{k}\gamma_{k})=\delta_{e,\gamma_0\gamma_1\ldots\gamma_{k}}\xi(\gamma_1,\ldots,\gamma_k)Tr(a_0\gamma_0(a_1)\ldots (\gamma_0\gamma_1\ldots\gamma_{k-1})(a_k)).
\end{equation}

The action of $\Gamma$ on $\AH(M)$ induces (modulo $\hbar$) an
action of $\Gamma$ on $M$ by symplectomorphisms.
Let $\sigma$ be the ``principal symbol" map:
\[
\AH(M) \rightarrow \AH(M)/\hbar \AH(M)\simeq
C^{\infty}(M).
\]
It induces a homomorphism
\begin{equation*}
\sigma \colon \AH(M)\rtimes \Gamma \longrightarrow
C^{\infty}(M)\rtimes \Gamma,
\end{equation*}
still denoted by $\sigma$. Let
\begin{equation*} \Phi\colon H^\sbullet_\Gamma(M) \longrightarrow HC_{per}^\sbullet\left(C_c^{\infty}(M) \rtimes \Gamma\right)
\end{equation*}
be the canonical map (first constructed by Connes) induced by \eqref{eq:Phi}, where
 $H^\sbullet_\Gamma(M)$ denotes the cohomology of the Borel construction $M\times_\Gamma E\Gamma$ and $C_c^\infty(M)$
denotes the algebra of compactly supported smooth functions on $M$.

The main result of this paper is the following.
\begin{theorem}\label{Mainresult}
  Let $e$, $f \in M_N\left(\AG\right) $ be a couple of idempotents such that the difference $e-f \in M_N\left(\AHc(M) \rtimes \Gamma\right)$ is compactly supported, here $\AHc(M)$ denotes the ideal of compactly supported elements of $\AH(M)$. Let $[\xi]\in \CH^k(\Gamma, \C)$ be a group cohomology class.
Then $[e]-[f]$ is an element of $K_0(\AHc(M)\rtimes \Gamma)$ and its pairing with the cyclic cocycle $Tr_\xi$ is given by
\begin{equation}
<Tr_\xi,[e]-[f] >= \left\langle \Phi \left(\hat{A}_\Gamma e^{\theta_\Gamma}[\xi] \right), [\sigma(e)]-[\sigma(f)]\right\rangle.
\end{equation}
Here $\hat{A}_\Gamma \in H^\sbullet_\Gamma(M)$ is the
equivariant $\hat{A}$-genus of $M$ (defined in section \ref{6}), $\theta_\Gamma \in H^\sbullet_\Gamma(M) $ is the equivariant characteristic class of the
deformation $\AH(M)$  (also defined in section \ref{6}).
\end{theorem}

In the case when the action of $\Gamma$ is free and proper, we recover the algebraic version of Connes-Moscovici higher index theorem.

The above theorem gives an algebraic version of the results of \cite{Sternin}, without the requirement that $\Gamma$ acts by isometries. To recover the analytic version of the index theorem type results from \cite{Sternin}  and\cite{Perrot} one can apply the methods  of \cite{N-T2}.
\subsection{Structure of the article}$\mbox{}$

\medskip

\noindent {\bf Section {\ref{section2}}} contains preliminary material, extracted mainly from \cite{RRII} and \cite{Globalana}. It is included for the convenience of the reader and contains the following material.

\medskip

\noindent $\bullet$ {\it Deformation quantization of symplectic manifolds and Gelfand--Fuks construction.}

\medskip

Following Fedosov, a deformation quantization of a symplectic manifold $\AH(M)$ can be seen as the space of flat sections of a flat connection $\nabla_F$ on the bundle $\mathcal{W}$ of Weyl algebras over $M$ constructed from the bundle of symplectic vector spaces $T^*M\rightarrow M$. The fiber of $\mathcal{W}$ is isomorphic to the Weyl algebra $\g=\mathbb{W}$ (see definition \ref{defWW}) and $\nabla_F$ is a connection with values in the Lie algebra of derivations of $\WW$, equivariant with respect to a maximal compact subgroup $K$ of the structure group of $T^*M$.

Suppose that $\mathbb{L}$ is a $(\g,K)$-module. The Gelfand--Fuks construction provides a complex $(\Omega (M,\mathbb{L}), \nabla_F )$ of $\mathbb{L}$-valued differential forms with a differential $\nabla_F$ satisfying $\nabla_F^2=0$. Let us denote  the corresponding spaces of cohomology classes by $H^*(M,\mathbb{L})$. An example is the Fedosov construction itself, in fact
$$
H^k(M,\WW)=
\begin{cases} \AH(M)&k=0\\0&k\neq 0.
\end{cases}
$$
The Gelfand--Fuks construction also provides a
 morphism of complexes
$$
GF\colon C^*_{Lie}(\g,K;\mathbb{L}))\longrightarrow \Omega (M,\mathcal{L})
$$

\medskip

\noindent{$\bullet$ \it  Algebraic index theorem}

\medskip

The Gelfand--Fuks map  is used to reduce the algebraic index theorem for a deformation of $M$ to its Lie algebra version involving only the $(\g,K)$-modules given by the periodic cyclic complexes of $\WW$ and the commutative algebra $\mathbb{O}=\C\llbracket  x_1,\ldots,x_n,\xi_1,\ldots,\xi_n,\hbar\rrbracket$.  In fact, the following holds.

\begin{theorem}
Let
$
 \mathbb{L}^\sbullet=\Hom^{-\sbullet}(CC^{per}_\sbullet(\WW),\hat{\Omega}^{-\sbullet}[\hbar^{-1},\hbar\rrbracket[u^{-1},u\rrbracket[2d]).
 $  There exist two elements
$\hat{\tau}_a$ and  $\hat{\tau}_t$ in the hypercohomology group $\mathbb{H}^0_{Lie}(\g,K;\mathbb{L}^\sbullet )$ such that the following holds.
 \begin{enumerate}
 \item Suppose that  $M=T^*X$ for a smooth compact manifold $X$ and $\AH(M)$ is the deformation coming from the calculus of differential operators.  Then  whenever $p$ and $q$ are two idempotent pseudodifferential operators  with $p-q$ smoothing,
  \begin{footnotesize}
 $$
 \int_M GF(\hat{\tau}_a)(\sigma(p)-\sigma(q))=Tr(p-q) \mbox{ and }
 \int_M GF(\hat{\tau}_t)(\sigma(p)-\sigma(q))=\int_M ch(p_0)-ch(q_0)
 $$
  \end{footnotesize}
 where $Tr$ is the standard trace on the trace class operators on $L^2(M)$, $p_0$ (resp. $q_0$) are the $\hbar=0$ components of $p$ (resp. $q$) and $ch\colon K^0(M)\rightarrow H^{ev}(M)$ is the classical Chern character.
\item
\begin{equation*}
\hat{\tau}_a = \sum_{p\geq 0}\left[\hat{A}_fe^{\hat{\theta}}\right]_{2p}u^p\hat{\tau}_t,
\end{equation*}
where $\left[\hat{A}_fe^{\hat{\theta}}\right]_{2p}$ is the component of degree $2p$ of a certain hypercohomology class. In the case of $M=T^*X$ as above, $GF\left(\hat{A}_fe^{\hat{\theta}}\right)$ coincides with the $\hat{A}$-genus of $M$.
\end{enumerate}
\end{theorem}

\medskip

\noindent {In general, given a group $\Gamma$ acting on a deformation quantization algebra $\AH(M)$, there does not exist any invariant Fedosov connection. As a result, the Gelfand--Fuks map described in section \ref{section2} does not extend to this case. The rest of the paper is devoted to the construction of a Gelfand--Fuks map that avoids this problem and the proof of the main theorem.}

\medskip

   {\bf Section \ref{3}} is devoted to a generalization of the Gelfand--Fuks construction to the equivariant case, where an analogue of the Fedosov construction and Gelfand--Fuks map are constructed on $M\times E\Gamma$.

   \medskip

    {\bf Section \ref{4}} is devoted to a  construction of a pairing of the periodic cyclic homology of the crossed product algebra with a certain Lie algebra cohomology appearing in Section \ref{section2}. The main tool is for this construction is  the Gelfand--Fuks maps from Section \ref{3}.

    \medskip

    {\bf Section \ref{6}} contains the proof of the main result.

    \medskip

    The appendix is used to define and prove certain statements about the various cohomology theories appearing in the main body of the paper. All the results and definitions in the appendices are well-known and standard and are included for the convenience of the reader.

\section{Algebraic Index Theorem}\label{section2}

\subsection{Deformed formal geometry}
Let us start in this section by recalling the adaptation of the framework of Gelfand-Kazhdan's formal geometry to deformation quantization described in
\cite{AIT, N-T99} and \cite{RRII}.

\medskip

\noindent {\it For the rest of this section we fix   a symplectic manifold $\left(M,\omega\right)$ of dimension $2d$ and   its deformation quantization $\AH\left(M\right)$.}

\medskip

\begin{notation} Let $m\in M$.
\begin{enumerate}
\item $\mathcal{J}_m^\infty(M)$ denotes the space of $\infty$-jets at $m\in M$; $\mathcal{J}_m^\infty(M):= \varprojlim C^\infty(M)/\left(\mathcal{I}_m\right)^k$, where
$\mathcal{I}_m$ is the ideal of smooth functions vanishing at $m$ and $k \in \N$.
\item Since the product in the algebra
$\AH\left(M\right)$ is local, it defines an associative, $\C\hhbar$-bilinear product $\star_m$ on $\mathcal{J}_m^\infty(M)$.
$\widehat{\AH\left(M\right)_m}$ denotes the algebra $(\mathcal{J}_m^\infty(M)\hhbar,\star_m)$.
\end{enumerate}
\end{notation}
\begin{notation} \label{Moyal}$\mbox{}$
\begin{enumerate}
\item   $\mathbb{W}$ will denote the algebra $\widehat{\AH\left(\R^{2d}\right)}_0$, where the deformation $\AH\left(\R^{2d}\right)$ has the product  given by the Moyal-Weyl formula
\begin{equation*}
\left(f\star g\right)\left(\xi,x\right)=\exp\left.\left(\frac{i\hbar}{2}\sum_{i=1}^d(\partial_{\xi^i}\partial_{y^i}-\partial_{\eta^i}\partial_{x^i})\right)f
\left(\xi,x\right)g\left(\eta,y\right)\right|_{\substack{\xi^i=\eta^i\\x^i=y^i}}.
\end{equation*}
\item Let $\hat{x}^k$, $\hat{\xi^k}$ denote the jets of $x^k$, $\xi^k$ -- the standard Darboux coordinates on $\R^{2d}$ respectively. $\mathbb{W}$ has a graded algebra structure, where the degree of the $\hat{x}^k$'s and $\hat{\xi}^k$'s is $1$ and the degree of $\hbar$ is $2$.
\item $\mathbb{W}$ will be endowed with the $\left\langle \hbar, \hat{x}^1,\ldots, \hat{x}^n,\hat{\xi}^1,\ldots,\hat{\xi}^n\right\rangle$-adic topology.
\item We denote the (symbol) map given by $\sum\hbar^kf_k\mapsto f_0$ by
\begin{equation*}
\hat{\sigma}_m\colon \widehat{\AH(M)}_m\longrightarrow \mathcal{J}_m^\infty(M).
\end{equation*}
  We shall also use the notation
$$
\mathcal{J}_0^\infty(\R^{2d})=:\mathbb{O}.
$$
\end{enumerate}
\end{notation}
\begin{definition}\label{defWW}
For a real symplectic vector space $(V,\omega)$ we denote
\begin{equation*}\WW(V):=\frac{\widehat{\mathcal{T}(V)}\otimes_\R\C\hhbar}
{\langle v\otimes w-w\otimes v-i \hbar\omega(v,w)\rangle}. \end{equation*}
Here $\mathcal{T}(V)$ is the tensor algebra of $V$, $\widehat{\mathcal{T}(V)}$ is its $V$-adic
completion and the topology is given by the filtration by assigning elements of $V$  degree $1$ and $\hbar$ degree $2$.
	\end{definition}
The assignment $V \mapsto \WW(V)$ is clearly functorial with respect to symplectomorphisms.

\begin{remark}
	Suppose $(V,\omega)$ is a $2d$-dimensional real symplectic vector space. A choice of symplectic basis for $V$ induces an isomorphism of $\C\hhbar$ algebras:
	\begin{equation*}
	\WW(V^*)\simeq \mathbb{W}.
	\end{equation*}
\end{remark}
\begin{notation}
 Let $\tG:=\Aut(\mathbb{W})$ denote the group of continuous $\Ch$-linear automorphisms of $\mathbb{W}$.
 We let $\g=\Der(\mathbb{W})$ denote the Lie algebra of continuous $\Ch$-linear derivations of $\mathbb{W}$.
\end{notation}
For future reference, let us state the following observation
\begin{lemma} The map
$$
\mathbb{W}\ni f\rightarrow \frac{1}{\hbar}\ad f\in \g
$$
is surjective. In particular, the grading of $\mathbb{W}$ induces a grading $\g=\prod_{i\geq -1}\g_i$ on $\g$, namely the unique grading such that this map is of degree $-2$ (note that $\WW_0=\C$ is central) . We will use the notation
$$
\g_{\geq k}=\prod_{i\geq k}\g_i.
$$
\end{lemma}
\begin{notation}
Let $\tg=\{ \frac{1}{\hbar}f\mid f\in \mathbb{W}\}$ with a Lie algebra structure given by
$$
[\frac{1}{\hbar}f,\frac{1}{\hbar}g]=\frac{1}{\hbar^2}[f,g]
$$
and note that $\tg$ is a central extension of $\g$. The corresponding short exact sequence has the form
\begin{equation}
\label{sesLieC}
0\longrightarrow\frac{1}{\hbar}\C\hhbar\longrightarrow \tg\stackrel{\ad}{\longrightarrow}\g \longrightarrow 0,
\end{equation}
where $\ad\frac{1}{\hbar}f(g)=\frac{1}{\hbar}[f,g]$.
\end{notation}

\noindent The extension \eqref{sesLieC} splits over $\lsp(2d,\C)$ and, moreover, the corresponding inclusion
$$
\lsp(2d,\C)\hookrightarrow \tg
$$
integrates to the action of $\Sp(2d,\C)$. The Lie subalgebra $\lsp(2d,\R)\subset \lsp(2d,\C)$ gets represented, in the standard basis, by elements of $\tg$ given by

$$
\frac{1}{\hbar}\hat{x}^k \hat{x}^j, \ \frac{-1}{\hbar}
\hat{\xi}^k \hat{\xi}^j \ \mbox{and} \ \frac{1}{\hbar}\hat{x}^k\hat{\xi}^j,\ \mbox{where} \ k,j=1,2,\ldots, d.
$$


\begin{lemma}
The Lie algebra $\g_{\geq0}$ has the structure of a semi-direct product
 \begin{equation*}
 \g_{\geq 0}= \g_{\geq 1}\rtimes\lsp(2d,\C).\end{equation*}
\end{lemma}
The  group $\tG$ of automorphisms of $\mathbb{W}$ has a structure of a pro-finite dimensional Lie group with the pro-finite dimensional Lie algebra $\g_{\geq0}$. As such, $\tG$ has the structure of semi-direct product
\begin{equation*}
\tG\simeq \tG_1\rtimes \Sp(2d,\C),
\end{equation*}
 where $\tG_1=\exp \g_{\geq 1}$ is pro-unipotent and contractible. The filtration of $\g$ induces a filtration of $\tG$:
\begin{equation*}
 \tG=\tG_0\supset \tG_1\supset \tG_2\supset \tG_3\supset \ldots.
 \end{equation*}
Here, for $k\geq 1$, $\tG_k=\exp(\g_{\geq k})$.

Recall that ${\AH\left(M\right)}$ is a formal deformation of a symplectic manifold $(M,\omega)$. In particular, $\omega$ gives an isomorphism $TM\rightarrow T^*M$ and the cotangent bundle $T^*M$ will be given the induced structure of a symplectic vector bundle.
For all $m\in M$, there exists a non-canonical isomorphism
$$
\widehat{\AH\left(M\right)}_m\simeq\mathbb{W}.
$$
We collect them together in the following.
\begin{definition}
 \begin{equation*}\tM:=\left\{\phi_m\mid m\in M, \phi_m\colon \widehat{\AH\left(M\right)}_m\stackrel{\sim}{\longrightarrow}\mathbb{W} \right\}
 \end{equation*}
The natural action of $\tG$ on $\tM$ endows it  with the structure of $\tG$-principal bundle. We give $\tM$ the pro-finite dimensional manifold structure using the pro-nilpotent group structure of  $\tG/Sp(2d,\C)$, see \cite{AIT} for details.
\end{definition}

\begin{theorem}\cite{AIT} The tangent bundle of $\tM$ is isomorphic to the trivial bundle $\tM\times\g$ and there exists a trivialisation
given by a $\g$-valued  one-form $\omega_\hbar\in\Omega^1\left(\tM\right)\otimes\g$ satisfying the Maurer-Cartan equation
\begin{equation*}
d\omega_\hbar+\frac{1}{2}[\omega_\hbar,\omega_\hbar]=0.
\end{equation*}
\end{theorem}
For later use let us introduce a slight modification of the above construction.
\begin{definition} Let $\tGr=\tG_1\rtimes \Sp(2d,\R)$. We will use  $\tMr$ to denote the $\tGr$-principal subbundle of $\tM$ consisting of the isomorphisms
\begin{equation*}
\phi_m\colon\widehat{\AH\left(M\right)}_m\stackrel{\sim}{\longrightarrow}\mathbb{W}
 \end{equation*}
 such that $\overline{\phi}_m$, the reduction of $\phi_m$ modulo $\hbar$,     is induced by a local symplectomorphism $(\R^{2d},0)\rightarrow (M,m)$.
\end{definition}
Note that the projection $\tMr\rightarrow M$ factorises through $\FM$, the bundle of symplectic frames in $TM$, equivariantly with respect to the action of  $\Sp(2d,\R)\subset \tGr$:
$$
\xymatrix{
\tGr\ar[r]&\tMr\ar[d]\\
\Sp(2d,\R) \ar[r]&\FM\ar[d]\\
&M.}
$$
We will use the same symbol for $\omega_\hbar$ and its pull back to $\tMr$.
\subsection{Fedosov connection and Gelfand-Fuks construction}
Recall that $\tG_1$ is contractible, thus in particular the principal $\tG_1$-bundle $\tMr\rightarrow \FM$ admits a  section $F$. Since the space $K$ is solid \cite{steenrod},  we can choose $F$ to be $\Sp(2d,\R)$-equivariant. Set
\begin{equation*}
A_F=F^*\omega_\hbar\in\Omega^1\left(\FM;\g\right).
\end{equation*}
Since $A_F$ is $\Sp(2d,\R)$-equivariant and satisfies the Maurer-Cartan equation,
\begin{equation}
d+A_F
\end{equation} reduces to  a flat $\g$-valued connection $\nabla_F$ on $M$, called the
\emph{Fedosov connection}.
\begin{example}\label{loctriv}
Consider  the case of  $M=\R^{2d}$ with the standard symplectic structure and let $\AH(\R^{2d})$ denote the Moyal-Weyl deformation. Then both
$\mathcal{F}_{\R^{2d}}$ and $\widetilde{\R^{2d}}$ are trivial bundles. The trivialization is given by the standard (Darboux) coordinates
$x^1,\ldots, x^d,\xi^1,\ldots, \xi^d$. So we see, using the construction of $\omega_\hbar$ in \cite{AIT}, that $A_F(X)=\frac{1}{i\hbar}[\omega(X,-),-]$, where we
consider $\omega(X,-)\in \Gamma(T^*M)\hookrightarrow \Gamma(M;\WW)$. 
 Let us denote the
generators of $\WW $ corresponding to the standard coordinates by $\hat{x}^i$ and $\hat{\xi}^i$, then we see that \begin{equation*}A_F(\partial_{x^i})=-\partial_{\hat{x}^i}\hspace{0.3cm}\mbox{and}
                                                                                                    \hspace{0.3cm} A_F(\partial_{\xi^i})=-\partial_{\hat{\xi}^i}.
                                                                                                   \end{equation*}
\end{example}

\begin{notation}
Suppose that  $\mathfrak{l}\subset\mathfrak{h}$ is an inclusion of Lie algebras and suppose that the ad action of $\mathfrak l$ on $\mathfrak{h}$ integrates to an action of a Lie group $L$ with Lie algebra $\mathfrak{l}$.  An $\mathfrak{h}$ module $\mathbb{M}$ is said to be an $(\mathfrak{h},L)$-module if
the  action of $\mathfrak{l}$ on $\mathbb{M}$ integrates to a compatible  action of the Lie group $L$. If an $(\mathfrak{h},L)$-module is equipped with a compatible grading and differential we will call it an $(\mathfrak{h},L)$-cochain complex.
\end{notation}

\begin{definition}
We set
 \begin{equation*}
 \Omega^\sbullet(M;\mathbb{L}):=\left\{\eta\in(\Omega^\sbullet(\FM)\otimes\mathbb{L})^{\Sps(2d)}\mid \iota_X(\eta)=0\hspace{0.2cm} \forall X\in\lsp(2d)\right\}
 \end{equation*}
  for a $(\g,\Sp(2d,\R))$-module $\mathbb{L}$.
 Here the superscript refers to taking invariants for the diagonal action and $\iota_X$ stands for contraction with the vertical vector fields tangent to the action of $\Sp(2n,\R)$.

 \noindent Together with $\nabla_F$,
 $
(  \Omega^\sbullet(M;\mathbb{L}),\nabla_F)
 $
 forms a cochain complex. The same construction with a $(\g,\Sp(2d,\R)$-cochain complex $(\LL^{\bullet},\delta)$ yields the double complex
  $
 (  \Omega^\sbullet(M;\mathbb{L}^\bullet),\nabla_F, \delta)
 $
\end{definition}

\begin{remark}
$\Omega^0(M;\mathbb{L})$ is the space of sections of a bundle which we will denote by $\mathcal{L}$, whose fibers are isomorphic to $\mathbb{L}$. $
(  \Omega^\sbullet(M;\mathbb{L}),\nabla_F)
 $ is the de Rham complex of differential forms with coefficients in $\mathcal{L}$.
\end{remark}

\begin{definition}\label{GF}
Suppose that $(\mathbb{L}^{\bullet },\delta) $ is a $(\g,\Sp(2d))$-cochain complex. The Gelfand-Fuks map $C_{Lie}^\sbullet(\g,\lsp(2d);\mathbb{L}^{\sbullet }) \longrightarrow \Omega^\sbullet(M;\mathbb{L}^{\sbullet })$
 is defined as follows.
Given $\phi\in C_{Lie}^n(\g,\lsp(2d);\mathbb{L}^{\bullet })$ and   vector fields $\{X_i\}_{i=1,\ldots,n}$ on $\FM$ set
 \begin{equation*}
 GF(\phi)(X_1,\ldots, X_n)(p)=\phi(A_F(X_1)(p),\ldots, A_F(X_n)(p)).
 \end{equation*}
\end{definition}
Direct calculation  using the fact that $\omega_{\hbar}$ satisfies the Maurer-Cartan equation gives the following theorem:
\begin{theorem}
The map $GF$ is a morphism of double complexes
 \begin{equation*}
 GF\colon \left(C_{Lie}^\sbullet(\g,\lsp(2d);\mathbb{L}^{\sbullet }),\partial_{Lie}, \delta\right)\longrightarrow \left(\Omega^\sbullet(M;\mathbb{L}^{\sbullet }),\nabla_F, \delta\right).
 \end{equation*}
The change of Fedosov connection, i.e. of the  section $F$, gives rise to a chain homotopic morphism of the total complexes.
\end{theorem}

\begin{example}\label{JF}
	
	\leavevmode
	
\begin{enumerate}
\item Suppose that $\mathbb{L}=\C$. The associated complex is just the de Rham complex of $M$.
\item Suppose that $\mathbb{L}=\mathbb{W}$. The associated bundle  $\WW(T^*M)$, the Weyl bundle of $M$, is given by applying the functor $\mathbb{W}$ to the symplectic vector bundle $T^*M$. Moreover, the choice of $F$ determines a canonical quasi-isomorphism
\begin{equation*}
J^\infty_F\colon (\AH,0)\longrightarrow \left(\Omega^\sbullet(M;\WW),\nabla_F\right).
\end{equation*}
\item Suppose that $\mathbb{L}=(CC^{per}_\sbullet (\WW), b+uB)$, the cyclic periodic complex of $\WW$. The complex  $\left(\Omega^\sbullet(M;CC^{per}_\sbullet (\WW)),\nabla_F+ b+uB\right)$ is a resolution of  the jets at the diagonal of the cyclic periodic complex of $\AH(M)$.

\end{enumerate}

\end{example}
\begin{example}
	
	\leavevmode
	
\begin{enumerate}
\item Let $\hat{\theta}\in C_{Lie}^2(\g,\lsp(2d);\C)$ denote a representative of the class of the extension \eqref{3}. The class of $\theta=GF(\hat{\theta})$ belongs to $\frac{\omega}{i\hbar}+\CH^2(M;\C)\hhbar$ and  classifies the deformations of $M$ up to gauge equivalence (see e.g. \cite{N-T99}).

\item The action of $\lsp(2d)$ on $\g$ is semisimple and $\lsp(2d)$ admits a $\Sp(2d,\R)$-equivariant complement. Let $\Pi$ be the implied $\Sp(2d,\R)$-equivariant projection $\g\rightarrow \lsp(2d)$.
Set $
 R\colon \g\wedge \g\longrightarrow \lsp(2d)
$ to be the two-cocycle
$$
R(X,Y)=[\Pi(X),\Pi(Y)]-\Pi([X,Y]).
$$
{\it The Chern-Weil homomorphism} is the map
\[
 CW\colon S^\sbullet(\lsp(2d)^*)^{\Sps(2d)}\longrightarrow \mathbb{H}^{2\sbullet}_{Lie}(\g,\lsp(2d))\]
 given on the level of cochains by
 \begin{equation*}
 CW(P)(X_1,\ldots, X_n)=P(R(X_1,X_2),\ldots, R(X_{n-1},X_n)).
 \end{equation*}
An example is the $\hat{A}$-power series
\begin{equation*}\hat{A}_f=CW\left( \det\left(\frac{ad(\frac{X}{2})}{\exp(ad(\frac{X}{2}))-\exp(ad(-\frac{X}{2}))}\right)\right).\end{equation*}
$GF(\hat{A}_f)=\hat{A}(TM)$, the $\hat{A}$-genus of the tangent bundle of $M$.
\end{enumerate}

\end{example}

\subsection{Algebraic index theorem in Lie algebra cohomology}

\begin{notation}
We denote
$\WW_{(\hbar)}:=\WW[\hbar^{-1}]$.
\end{notation}
\begin{notation}
Our convention for shifts of complexes is as follows:
 $$
 \left(V^{\sbullet}[k]\right)^p=V^{p+k}.
 $$
\end{notation}

\begin{theorem}[\cite{RRII},\cite{NCDG}]\label{th:classes} Let $(\hat{\Omega}^\sbullet,\hat{d})$ denote the formal de Rham complex in $2d$ dimensions, and let $\left(C_\sbullet\left(\WWHl\right),b\right)$ denote the Hochschild complex of $\WWHl$.
\begin{enumerate}
\item There exists a unique (up to homotopy) quasi-isomorphism
  \begin{equation*}
  \mu^\hbar\colon \left(C^{Hoch}_\sbullet\left(\WWHl\right),b\right)\longrightarrow \left(\hat{\Omega}^{-\sbullet}[\hbar^{-1},\hbar\rrbracket[2d],\hat{d}\right).
  \end{equation*}
 which maps the Hochschild
 $2d$-chain
$$
\varphi = 1 \otimes \alt \left(\hat{\xi}_{1}\otimes \hat{x}_{1}\otimes \hat{\xi}_{2}\otimes \hat{x}_{2}\otimes\ldots\otimes \hat{\xi}_{d}\otimes \hat{x}_{d} \right),
 $$
where
$\alt (z_1 \otimes \ldots \otimes z_n):
= \sum_{\sigma\in \Sigma_{n}}(-1)^{sgn \sigma} z_{\sigma(1)} \otimes \ldots \otimes z_{\sigma(n)}$,
 to the $0$-form $1$.  $\mu^\hbar$ extends to a quasi-isomorphism
  \begin{equation*}
  \mu^\hbar\colon (CC^{per}_\sbullet(\WWHl),b+uB)\longrightarrow (\hat{\Omega}^{-\sbullet}[\hbar^{-1},\hbar\rrbracket[u^{-1},u\rrbracket[2d],\hat{d}).
  \end{equation*}
  \item The principal symbol map $\sigma \colon \WW\rightarrow \WW/\hbar\WW\simeq \mathbb{O}$ together with the Hochschild-Kostant-Rosenberg map $HKR$
 given by
 $$
 f_0\otimes f_1\otimes \ldots\otimes f_n\mapsto \frac{1}{n!}f_0\hat{d}f_1\wedge \hat{d}f_2\wedge\ldots\wedge \hat{d}f_n
 $$
 induces a $\C$-linear quasi-isomorphism
\begin{equation*}
\hat{\mu}\colon CC^{per}_\sbullet(\WW)\longrightarrow \left(\hat{\Omega}^\sbullet[u^{-1},u\rrbracket, u\hat{d}\right).
\end{equation*}
\item The map of complexes $
J\colon (\hat{\Omega}^\sbullet[u^{-1},u\rrbracket, u\hat{d})\rightarrow (\hat{\Omega}^{-\sbullet}[\hbar^{-1},\hbar\rrbracket[u^{-1},u\rrbracket[2d], \hat{d})$  given by
$$
f_0\hat{d}f_1\wedge\ldots\wedge \hat{d}f_n\mapsto u^{-d-n}f_0\hat{d}f_1\wedge\ldots\wedge \hat{d}f_n.
$$
makes the following diagram commute up to homotopy,.
\begin{equation}
\xymatrix{&CC_\sbullet^{per}(\WW)\ar[r]^{\iota}\ar[d]^\sigma
&CC_\sbullet^{per}(\WWHl)\ar[r]^>>>{\mu^\hbar}&(\hat{\Omega}^{-\sbullet}[\hbar^{-1},\hbar\rrbracket[u^{-1},u\rrbracket[2d],\hat{d})\\
&CC_\sbullet^{per}(\mathbb{O})\ar[rr]_{HKR}&&\left(\hat{\Omega}^\sbullet[u^{-1},u\rrbracket, u\hat{d}\right)\ar[u]_J
}.
\end{equation}
Here the complex $CC_\sbullet^{per}(\WW)$ at the leftmost top corner is that of $\WW$ as an algebra over $\C$:
\end{enumerate}
 \end{theorem}
 \begin{remark}One can in fact extend the above $\C$-linear "principal symbol map"
 $$
 \sigma\colon CC_\sbullet^{per}(\WW)\rightarrow CC_\sbullet^{per}(\mathbb{O})
 $$
to a $\C\hhbar$-linear map of complexes
 $
 CC_\sbullet^{per}(\WW)\rightarrow CC_\sbullet^{per}(\mathbb{O}\hhbar),
 $
 but we will not need it below.
 \end{remark}
\begin{notation}
\begin{enumerate}
\item Action of $\g$ by derivations on $\WW$ extends to the complex $CC^{per}_\sbullet(\WW)$ and we give it the corresponding $(\g,\Sp(2d,\R))$-module structure.
\item The action of $\g$ on $\WW$ taken modulo  $\hbar \WW$, induces an action of $\g$ (by Hamiltonian vector fields)
on $(\hat{\Omega}^{-\sbullet},d)$ and hence on $(\hat{\Omega}^{-\sbullet}[\hbar^{-1},\hbar\rrbracket[u^{-1},u\rrbracket[2d],d)$. We give $(\hat{\Omega}^{-\sbullet}[\hbar^{-1},\hbar\rrbracket[u^{-1},u\rrbracket[2d],d)$ the induced structure of $(\g,\Sp(2d,\R))$-module.
\item
 We set
 \begin{equation*}
 \mathbb{L}^\sbullet:=\Hom^{-\sbullet}(CC^{per}_\sbullet(\WW),\hat{\Omega}^{-\sbullet}[\hbar^{-1},\hbar\rrbracket[u^{-1},u\rrbracket[2d]).
 \end{equation*}
 $\mathbb{L}$ inherits the $(\g,\Sp(2d,\R))$-module structure from the actions of $\g$ described above.
\end{enumerate}

\end{notation}

The composition $J\circ{HKR}\circ\hat{\sigma}$ is  equivariant with respect to the action of $\g$, hence the following definition makes sense.
 \begin{definition}
 $\hat{\tau_t}$ is the cohomology class in the hypercohomology $\mathbb{H}^0_{Lie}\left(\g,\lsp(2d);\mathbb{L}_{(\hbar)}\right)$ given by the cochain
\begin{equation}
J\circ {HKR}\circ {\sigma}\in C_{Lie}^0(\g,\Sp(2d,\R);\mathbb{L}^0).
\end{equation}
\end{definition}

\begin{lemma} The cochain
$$
\mu^\hbar\circ \iota \in C_{Lie}^0(\g,\Sp(2d,\R);\mathbb{L}^0)
$$
extends to a cocycle  in the complex
$$
( C_{Lie}^\sbullet(\g,\Sp(2d,\R);\mathbb{L}^\sbullet),\partial_{Lie}+\partial_{\mathbb{L}}).
$$
The cohomology class of this cocycle is independent of the choice of the extension. We will denote the corresponding class in $\mathbb{H}^0_{Lie}\left(\g,\lsp(2d);\mathbb{L}_{(\hbar)}\right)$ by $\hat{\tau}_a$.
\end{lemma}
For a proof of the next result see e.g. \cite{RRII}.
\begin{theorem}[Lie Algebraic Index Theorem]\label{UAIT}
We have
 \begin{equation*} \hat{\tau}_a = \sum_{p\geq 0}\left[\hat{A}_fe^{\hat{\theta}}\right]_{2p}u^p\hat{\tau}_t,\end{equation*}
where $\left[\hat{A}_fe^{\hat{\theta}}\right]_{2p}$ is the component of degree $2p$ of the cohomology class of $\hat{A}_fe^{\hat{\theta}}$ .
\end{theorem}

\subsection{Algebraic index theorem}\label{2.1}

An example of an application of the above is the algebraic index theorem for a formal deformation of  a symplectic manifold $M$.
 Note that we can view $\AH$ as a complex concentrated in degree $0$ and with trivial differential. Then, using the notation of remark \ref{JF}, we find the quasi-isomorphism
\begin{equation*}J^\infty_F\colon (\AH,0)\longrightarrow \left(\Omega^\sbullet(M;\WW),\nabla_F\right).\end{equation*} Similarly we find the quasi-isomorphism
\begin{equation*}J_F^\infty\colon CC^{per}_\sbullet(\AH)\longrightarrow \Omega^\sbullet(M;CC^{per}_\sbullet(\WW)).\end{equation*}
For future reference let us record the following observation.

\begin{lemma}\label{O=O}
	The quasi-isomorphic inclusion $\mathbb{C}[\hbar^{-1},\hbar\rrbracket[u^{-1},u\rrbracket  \hookrightarrow \hat{\Omega}^{-\sbullet}[\hbar^{-1},\hbar\rrbracket[u^{-1},u\rrbracket$ induces
a quasi-isomorphism
\[
\iota \colon \left(\Omega^\sbullet(M)[\hbar^{-1},\hbar\rrbracket[u^{-1},u\rrbracket[2d], d_{\dR}\right) \longrightarrow \left(\Omega^\sbullet(M;\hat{\Omega}^{-\sbullet}[\hbar^{-1},\hbar\rrbracket[u^{-1},u\rrbracket[2d]), \nabla_{F}+\hat{d}\right).
\]
\end{lemma}
\begin{notation}
We denote the inverse (up to homotopy) of $\iota$ by
	\begin{equation*}T_0\colon\left(\Omega^\sbullet(M;\hat{\Omega}^{-\sbullet}[\hbar^{-1},\hbar\rrbracket[u^{-1},u\rrbracket[2d]), \nabla_{F}+\hat{d}\right)\longrightarrow
	\left(\Omega^\sbullet(M)[\hbar^{-1},\hbar\rrbracket[u^{-1},u\rrbracket[2d], d_{\dR}\right).\end{equation*}

\end{notation}

For each $Q\in \Omega^\sbullet(M;\mathbb{L}^\sbullet)$ of total degree zero, we can define the map

\begin{multline*}C_Q\colon CC^{per}_0(\AH)\longrightarrow \Omega^\sbullet(M;CC^{per}_\sbullet(\WW))\stackrel{Q}{\longrightarrow}
 \Omega^\sbullet(M;\hat{\Omega}^{-*}[\hbar^{-1},\hbar\rrbracket[u^{-1},u\rrbracket[2d])\stackrel{T_0}\longrightarrow \\
   \Omega^{\sbullet-*}(M;\C)[\hbar^{-1},\hbar\rrbracket[u^{-1},u\rrbracket[2d]
 \stackrel{u^{-d}\int_M}{\longrightarrow}\C[\hbar^{-1},\hbar\rrbracket.
\end{multline*}
 Clearly $C_Q$ is a periodic cyclic cocycle if $Q$ is a cocycle. We will apply this construction to the two cocycles $\hat{\tau}_t$ and $\hat{\tau}_a$.

Let us start with $C_{\hat{\tau}_t}$.
Tracing the definitions we get the following result.
\begin{proposition}
$C_{\hat{\tau}_t}$ is given by 
\[
u^n w_0\otimes\ldots\otimes w_{2n}\mapsto  \frac{u^{n-d}}{(2n)!}\int_M \sigma(w_0){d}\sigma(w_1)\wedge\ldots\wedge {d}\sigma(w_{2n}).
\]
\end{proposition}
To get the corresponding result for $C_{\hat{\tau}_a}$ recall first that the algebra $\AH(M)$ has a unique $\C\hhbar$-linear trace, up to a normalisation factor. This factor can be fixed as follows.
Locally any deformation of a symplectic manifold is isomorphic to the Weyl deformation. Let $U$ be such a coordinate chart and let $\phi \colon \AH(U) \to \AH(\R^{2d})$ be an isomorphism. Then  the trace $Tr$ is normalized by requiring that for any $f \in \AHc(U) $ we have
$$
Tr(f)=\frac{1}{(i\hbar)^d}\int_{\R^{2d}}\phi(f)\frac{\omega^d}{d!}.
$$
\begin{proposition} We have
$$
C_{\hat{\tau}_a}=Tr.
$$
\end{proposition}
\begin{proof}[Comments about the proof.]

\noindent
First one checks that $C_{\hat{\tau}_a}$ is a $0$-cocycle and therefore a trace.
Hence it is a $\C \hhbar$-multiple of $Tr$ and it is sufficient to evaluate it
on elements supported in a coordinate chart. Moreover,  the fact that the Hochschild cohomology class of $C_{\hat{\tau}_a}$ is independent of the Fedosov connection implies that $C_{\hat{\tau}_a}$ is independent of it. Thus it is sufficient to verify the statement for $\R^{2d}$ with the standard Fedosov connection. Let $f \in \AHc(\R^{2d})$,
one checks that $J^\infty_F(f) \in \Omega^0(\R^{2d}; C_0(\WW))$ is cohomologous to the
element $\frac{1}{(i\hbar)^d d!} f\, \varphi\, \omega^d \in \Omega_c^{2d}(\R^{2d}; C_{2d}(\WW))$ in $\Omega_c^{\sbullet}(\R^{2d}; C_{-\sbullet}(\WW))$ .
It follows that the  $GF(\hat{\tau}_a) J^\infty_F(f)$ is cohomologous to $\frac{1}{(i\hbar)^d d!} f \, \varphi\, \omega^d$ (see the theorem \ref{th:classes}) and therefore
\[
C_{\hat{\tau}_a}(f) = Tr(f)=\frac{1}{(i\hbar)^d}\int_{\R^{2d}}f \frac{\omega^d}{d!}
\]
and the statement follows.

\end{proof}

Given above identifications of $C_{\hat{\tau}_a}$ and $C_{\hat{\tau}_t}$, the theorem \ref{UAIT} implies the following result.

\begin{theorem}[Algebraic Index Theorem]\label{thm:AIT}
 Suppose $a\in CC^{per}_0\left(\AHc\right)$ is a cycle, then
 \begin{equation*}
 Tr(a)=u^{-d}\int_M \sum_{p\geq 0}HKR(\sigma(a))\left(\hat{A}(T_\C M)e^{\theta}\right)_{2p}u^p.
 \end{equation*}
 \end{theorem}

\section{Equivariant Gelfand-Fuks map}\label{3}

Suppose $\Gamma$ is a discrete group acting by automorphisms on $\AH$. Then we can extend the action to $\tMr$ as follows. First note that the action on
$\AH$ induces an action on $\left(M,\omega\right)$. Now suppose
$\left(m,\phi_m\right)\in\tMr$ and $\gamma\in\Gamma$, then let $\gamma\left(m,\phi_m\right)=\left(\gamma\left(m\right),\phi_m^\gamma\right)$,
here $\phi_m^\gamma$ is given by \begin{equation*}\widehat{\AH\left(M\right)}_{\gamma\left(m\right)}\longrightarrow\widehat{\AH\left(M\right)}_m\longrightarrow\mathbb{W},\end{equation*} where the first
arrow is given by the action of $\Gamma$ on $\AH$ and the second arrow is given by $\phi_m$. Note that, since $\tGr$ acts on $\tMr$ by
postcomposition and $\Gamma$ acts on $\tMr$ by precomposition, we find that the two actions commute.

It should be apparent from the preceding section that we will need an equivariant version of the Gelfand-Fuks map in the deformed setting. This will allow us to derive
the
equivariant algebraic index theorem from the Lie algebraic one. To do this we will extend the definition of the one-form $A_F$ to the Borel construction
$E\Gamma\times_\Gamma M$. Explicitly this
is done by defining connection one-forms $A_{Fk}$ on the manifolds $\Delta^k\times\Gamma^k\times\FM$ which satisfy certain boundary conditions. This connection
one-form will then serve the usual purpose in the Gelfand-Fuks map, only now the range of the map will be
a model for the equivariant cohomology of the manifold.

Assume that $\Gamma$ is a discrete group acting on a manifold $X$. Set $X_k: = X \times \Gamma^k $. Define the face maps $\partial_i^k \colon X_k \to X_{k-1}$ by
\[
 \partial_i^k(x, \gamma_1,\ldots,\gamma_k)=\begin{cases}(\gamma^{-1}_1(x), \gamma_2,\ldots,\gamma_k)&\text{if } i=0\\
                                  (x, \gamma_1,\ldots, \gamma_i\gamma_{i+1},\ldots,\gamma_k)& \text{if } 0<i< k\\
                                  (x,\gamma_1,\ldots, \gamma_{k-1})&\text{if } i=k\\
                                 \end{cases}
\]
 We denote the standard $k$-simplex by \begin{equation*}\Delta^k:=\left\{(t_0,\ldots,t_k)\geq 0|\sum_{i=0}^kt_i=1\right\}\subset \R^{k+1}\end{equation*} and define by
 $\epsilon_i^k\colon \Delta^{k-1}\to \Delta^k$
 \begin{equation*}\epsilon_i^k(t_0,\ldots,t_{k-1})=\begin{cases}(0,t_0,\ldots,t_{k-1})&\text{if } i=0\\
                                  (t_0,\ldots, t_{i-1},0,t_i,\ldots ,t_{k-1})& \text{if } 0<i\leq k
                                 \end{cases}\end{equation*}
\begin{definition} A \emph{de Rham-Sullivan}, or \emph{compatible} form $\phi$ of degree $p$ is a collection of forms $\phi_k \in \Omega^p(\Delta^k \times X_k )$, $k=0, 1, \ldots$, satisfying
 \begin{equation}
\label{eq:compatibility conditions} (\epsilon_i^k\times \id )^*\phi_k=(\id \times \partial_i^k)^*\phi_{k-1} \in \Omega^p(\Delta^{k-1}\times X_k)\end{equation} for $0\le i\le k$ and any $k>0$.
 \end{definition}
If  $\phi = \{\phi_k\}$ is a compatible form, then $d \phi := \{d \phi_k\}$ is also a compatible form;
for two compatible forms $\phi = \{\phi_k\}$ and $\psi = \{\psi_k\}$ their product $\phi\psi :=\{\phi_k \wedge\psi_k\}$ is another compatible form.
We denote the space of de Rham-Sullivan forms by $\Omega^{\sbullet}(M\times_\Gamma E\Gamma)$ in view of
 the following
\begin{theorem}\label{Borel=Borel}
We have
 \begin{equation*}\CH^\sbullet \left( \Omega^{\sbullet}(M\times_\Gamma E\Gamma), d \right) \simeq \CH^\sbullet_\Gamma(M) \end{equation*}
 where the left hand side is the cohomology of  the complex $\Omega^{\sbullet}(M\times_\Gamma E\Gamma)$ and  the right hand side
 is the cohomology of the Borel construction $M\times_\Gamma E\Gamma$.
\end{theorem}

 \noindent See for instance \cite{Dupont} for the proof.

More generally, let $V$ be a $\Gamma$-equivariant bundle on $X$. Let $\pi_k \colon X_k \to X$ be the projection and let $V_k: = \pi_k^* V$. Notice that we have canonical isomorphisms
\begin{equation}\label{isoV}
(\partial_i^k)^* V_{k-1} \cong V_k.
\end{equation}
\begin{definition}
Let $V$ be  $\Gamma$-equivariant vector bundle.
A $V$ valued de Rham-Sullivan (compatible) form $\phi$ is a collection  $\phi_k \in \Omega^p(\Delta^k \times X_k; V_k )$, $k=0, 1, \ldots$, satisfying the conditions \eqref{eq:compatibility conditions}, where we use the isomorphisms \eqref{isoV} to identify $(\partial_i^k)^* V_{k-1}$ with $V_k$.

We let $\Omega^{\sbullet}(M\times_\Gamma E\Gamma; V)$ denote the space of $V$-valued de Rham-Sullivan (compatible) forms.
\end{definition}
For equivariant vector bundles $V$ and $W$ there is a product
\[ \Omega^{\sbullet}(M\times_\Gamma E\Gamma; V) \otimes  \Omega^{\sbullet}(M\times_\Gamma E\Gamma; W) \longrightarrow \Omega^{\sbullet}(M\times_\Gamma E\Gamma; V\otimes W)\]
defined as for the scalar forms by $\phi\psi :=\{\phi_k \wedge\psi_k\}$.

Assume that we have a collection of connections $\nabla_k$ on the bundles $V_k$ satisfying the compatibility conditions
\begin{equation}\label{connectioncompatibility}
 (\epsilon_i^k\times \id )^*\nabla_k=(\id \times \partial_i^k)^*\nabla_{k-1}.
\end{equation}
Then for a compatible form $\phi = \{\phi_k\}$
\[
\nabla \phi:= \{ \nabla_k \phi_k\}
\]
is again a compatible form.

\begin{notation}
 Now let $M$ be a symplectic manifold and $\Gamma$ a discrete group acting by symplectomorphisms on $M$.
We introduce the following notations:
 \begin{equation*}P^k_\Gamma:=\Delta^k\times(\FM)_k=\Delta^k\times\FM\times\Gamma^k, \end{equation*}
 and similarly
 \begin{equation*}M^k_\Gamma:=\Delta^k\times M_k=\Delta^k\times M\times\Gamma^k\end{equation*} and
 \begin{equation*}\tM^k_\Gamma:=\Delta^k\times (M_r)_k=\Delta^k\times \tMr\times \Gamma^k.\end{equation*} Note that $P^k_\Gamma\rightarrow M^k_\Gamma$ is a principal $\Sp(2d)$-bundle, namely the pull-back of $\FM\rightarrow M$ via the
 obvious projection. Similarly $\tM^k_\Gamma$ is the pull-back of $\tMr\rightarrow M$. We define
 $\Omega^\sbullet(M^k_\Gamma; \mathbb{L})$ for a $(\g,\Sp(2d))$-module $\mathbb{L}$ as we did for $M$ above only replacing $\FM$ by $P^k_\Gamma$ and considering
 the trivial action of the symplectic group on $\Delta^k\times\Gamma^k$. We shall denote the $\tG_1$-principal bundle $\tMr\rightarrow \FM$ by $\pi_1$,
 the $\tGr$-principal bundle $\tMr\rightarrow M$ by $\pi_r$ and the $\Sp(2d)$-principal bundle $\FM\rightarrow M$ by $\pi$.
 \end{notation}

Note that, for all $i$ and $k$, the (obvious) fibration of $\Delta^{k-1}\times(\tM_r)_k$ over $\Delta^{k-1}\times(\FM)_k$ is canonically isomorphic to the pull back by $\id\times \partial^k_i$ of the fibration of $\tM_\Gamma^{k-1}$ over $P_\Gamma^{k-1}$. Hence, for any section $F$ of the projection $\tM_\Gamma^{k-1}\rightarrow P_\Gamma^{k-1}$, there is a natural pull back
$
(\id\times \partial^k_i)^*F,
$  the unique section of $\Delta^{k-1}\times(\tM_r)_k\rightarrow \Delta^{k-1}\times(\FM)_k$
making the following diagram commutative:
\begin{equation}
\xymatrix{\Delta^{k-1}\times(\tM_r)_k\ar[r]^{\hspace{0.5cm}\id\times \partial^k_i}&\tM_\Gamma^{k-1}\\
\Delta^{k-1}\times(\FM)_k\ar[r]^{\hspace{0.5cm}\id\times \partial^k_i}\ar[u]^{(\id\times \partial^k_i)^*F}&P_\Gamma^{k-1}\ar[u]_F
}
.
\end{equation}

 \begin{lemma}\label{bc}
 There exist sections
 $$
 \xymatrix{\tM_\Gamma^k\ar[d]\\
P_\Gamma^k\ar@/^2pc/[u] ^{F_k}}
 $$
 of the projections
 $
 \tM_\Gamma^k\rightarrow P_\Gamma^k
 $
 satisfying the compatibility conditions
   \begin{equation*}(\epsilon_i^k\times \id )^*F_k=(\id \times \partial_i^k)^*F_{k-1}
   \end{equation*}

\end{lemma}
\begin{proof}

 \leavevmode

 \noindent We construct the sections $F_k$ recursively. In section \ref{section2} we constructed the inital section $F_r\colon P_\Gamma^0=\FM\longrightarrow \tMr$. Set $F_0:=F_r$.

Now suppose we have found $F_l$ satisfying the compatibility conditions for all $l<k$.
Notice that the principal bundle $\tM_\Gamma^k\rightarrow P_\Gamma^k$ is trivial, and thus, by fixing a trivialization, we can view its sections as functions
on $P_\Gamma^k$ with values in $\tG_1$. which we identify with a vector space $\g_1$ via the exponential map.
The compatibility conditions require that $F_k$ takes on certain values determined by $F_{k-1}$ on $(\partial \Delta^k) \times \FM\times \Gamma^k \subset P_\Gamma^k$ .
Since $\tG_1$ can be identified with a vector space $\g_1$ via the exponential map, $F_k$ can be extended smoothly from  $(\partial \Delta^k) \times \mathcal{F}\times \Gamma^k$ to $P_\Gamma^k$

\end{proof}

As before we can construct the sections in the lemma above $\Sp(2d)$ equivariantly and we will fix a system of such equivariant sections $\{F_k\}_{k\geq0}$ from now on.
Now,  as before,  we can use the sections $F_k$ to pull back the canonical
connection form from $\tM_\Gamma^k$ (which was itself pulled back from  $\tM$ through the composition of the projection onto $\tMr$ and an inclusion), to define a
$\g$-valued differential form $A_{Fk}$ on $P^k_\Gamma$ for each $k$.

\begin{notation}
 Suppose $\left(\mathbb{L}^\bullet,\partial_{\mathbb{L}}\right)$ is a $(\g,\tGr)$-cochain complex. Then  we denote by
 $\left(\mathbb{L}_{\pi_r}^\bullet,\partial_{\mathbb{L}}\right)$ the bundle of cochain
 complexes over $M$ associated to $\tMr$, i.e. with total space $\tMr\times_{\tGr}\mathbb{L}$. We will denote the pull-back to the
 $M_\Gamma^k$ by the same symbol. Note that the pullback $\pi^*\mathbb{L}_{\pi_r}$ is exactly the bundle associated to $\tMr\rightarrow\FM$
 with fiber the $(\g,\tG_1)$-cochain complex $\mathbb{L}$ given by $\tG_1\hookrightarrow \tGr$, i.e the pull-back has total space $\tMr\times_{\tG_1}\mathbb{L}$. Thus we will
 denote $\pi^*\mathbb{L}_{\pi_r}=\mathbb{L}_{\pi_1}$. Again we will use the same notation for the pull-backs over the $P_\Gamma^k$.

 \end{notation}

\begin{remark}
 Note that since $\Gamma$ acts on the $\tGr$-bundle $\tMr\rightarrow M$ we find that
 $\Gamma$ also acts on
 $\mathbb{L}_{\pi_r}^\sbullet$ and, since this action lifts to an $\Sp(2d)$-equivariant action on $\tMr\rightarrow\FM $, we find a corresponding action on the space
 $\Omega^\bullet\left(M;\mathbb{L}^\bullet\right)$. Let us be a bit more precise about this action. Note first that the section $F_0$ yields a trivialization
 of $\tMr\rightarrow \FM$. This means also that it yields a trivialization (denoted by the same symbol)
 \begin{equation*}F_0\colon \FM\times\mathbb{L}\stackrel{\sim}{\longrightarrow} \tMr\times_{G_1}\mathbb{L}\end{equation*} explicitly given by $(p,\ell)\mapsto [F_0(p),\ell]$ with the inverse given by mapping
 $[\phi_m,\ell]$ to \\$\left(\pi_1(\phi_m),\left(\phi_m\circ F_0(\pi_1(\phi_m))^{-1}\right)(\ell)\right)$. In these terms the action is given by
 \begin{equation*}\gamma(\eta\otimes\ell)=(\gamma^*\eta)\otimes (F_0^{-1})^*\gamma^*F_0^*\ell\end{equation*}
 where $(F_0^{-1})^*\gamma^*F_0^*\ell$ is the section given by
 \begin{equation*}p\mapsto (\gamma(p),F_0(p)\gamma F_0(\gamma(p))^{-1}\ell).\end{equation*} We consider the corresponding action of $\Gamma$ on the  spaces
 $\Omega^\bullet\left(M_\Gamma^k;\mathbb{L}^\bullet\right)$, where we use $F_k$ instead of $F_0$ (or in fact on
 $\Omega^\bullet\left(N\times M;\mathbb{L}^\bullet\right)$ for any $N$).
\end{remark}

Note that the differential forms $A_{F_k}$ define flat connections $\nabla_{F_k}$ on $\Omega^\bullet\left(M^k_\Gamma;\mathbb{L}^\bullet\right)$ for all $k$ and so we can
consider the product complex
$\prod_{k}\Omega^\bullet\left(M^k_\Gamma;\mathbb{L}^\bullet\right)$ with the differential $\tilde{\nabla}+\partial_{\mathbb{L}}$, where
$\tilde{\nabla}=\prod_k\nabla_{F_k}$. Note also that the connections $\nabla_{F_k}$ satisfy the compatibility conditions of the equation \eqref{connectioncompatibility}.

Now we can consider the equivariant Gelfand--Fuks map
\begin{equation*}GF_\Gamma\colon C_{Lie}^\bullet\left(\g,\lsp\left(2d\right);\mathbb{L}^\bullet\right)\longrightarrow\prod_{k}\Omega^\bullet\left(M^k_\Gamma;\mathbb{L}^\bullet\right)\end{equation*}
given by
\begin{equation}\label{GFG}GF_\Gamma\left(\chi\right)_k=\chi\circ A_{F_k}^{\otimes p},\end{equation}
where $\chi\in C_{Lie}^p\left(\g,\lsp\left(2d\right);\mathbb{L}^\bullet\right)$ and the subscript $k$ refers to taking the $k$-th coordinate in the product. In other
words the definition is the same as in definition \ref{GF} only we now use the compatible system of connections $A_{F_k}$.

\begin{lemma}\label{gfsdr}
 For all $\chi\in C^p_{Lie}(\g,\lsp(2d);\mathbb{L}^\sbullet)$ we have that $GF_\Gamma(\chi)\in \Omega^\sbullet\left(M\times_\Gamma E\Gamma, \mathbb{L}^\bullet\right)$.
\end{lemma}
\begin{proof}

 \leavevmode

 \noindent The boundary conditions put on the sections $F_k$ in lemma \ref{bc} are meant exactly to ensure this property of the Gelfand-Fuks map $GF_\Gamma$. The lemma
 follows straightforwardly from these boundary conditions.
\end{proof}

\begin{theorem}\label{gfg}

 The equivariant Gelfand-Fuks map   is a morphism of complexes
 \[
GF_\Gamma\colon  C_{Lie}^\bullet\left(\g,\lsp\left(2d\right);\mathbb{L}^\bullet\right) \rightarrow \Omega^\sbullet\left(M\times_\Gamma E\Gamma, \mathbb{L}^\bullet\right).
 \]

\end{theorem}
\begin{proof}

 \leavevmode

 \noindent This proof is exactly the same as in the non-equivariant setting, carried out coordinate-wise in the product. Since we do not give the usual
 proof in this article let us be a bit more explicit. Note that
 \begin{equation*}GF_\Gamma\colon C_{Lie}^\bullet\left(\g,\lsp\left(2d\right);\mathbb{L}^\bullet\right)\longrightarrow \prod_k\Omega^\bullet\left(M^k_\Gamma;\mathbb{L}^\bullet\right)\end{equation*}
 is given by \begin{equation*}GF_\Gamma\left(\chi\right)_{m_k}\left(X_1,\ldots X_p\right)=\chi\left(\left(A_{F_k}\right)_{m_k}X_1,\ldots, \left(A_{F_k}\right)_{m_k}X_p\right),\end{equation*}
 for $\chi\in C_{Lie}^p\left(\g,\lsp\left(2d\right);\mathbb{L}^\bullet\right)$, $m_k\in P^k_\Gamma$ and $X_i\in T_{m_k}P^k_\Gamma$. The differential
 $\partial_{\mathbb{L}}$ and $GF_\Gamma$
 clearly commute so it is left
 to show that \begin{equation*}GF_\Gamma\circ\partial_{Lie}=\tilde{\nabla}\circ GF_\Gamma.\end{equation*} This follows by direct computation using the facts that \begin{equation*}dA_{F_k}+\frac{1}{2}[A_{F_k},A_{F_k}]=0\end{equation*} and
 \begin{equation*}\tilde{\nabla}=\prod_k d_{\scriptstyle P^k_\Gamma}+A_{F_k},\end{equation*} here $d_{\scriptstyle P^k_\Gamma}$ refers to the de Rham differential on
 $P^k_\Gamma$.

\end{proof}


\section{Pairing with \texorpdfstring{$\HC_\bullet^{per}\left(\AHc\rtimes\Gamma\right)$}{HC(AG)}}\label{4}

In order to derive the equivariant version of the algebraic index theorem we should show that the universal class $\hat{\tau}_a$ maps to the class of the
trace supported at the identity under the equivariant Gelfand-Fuks map $GF_\Gamma$ constructed in the previous section. The class $\hat{\tau}_a$ lives in the Lie
algebra cohomology with values in the $(\g,\tGr)$-cochain complex
\begin{equation*}\mathbb{L}^\sbullet:=\mbox{Hom}^{-\sbullet}(CC^{per}_\sbullet(\WW_{\hbar}),\hat{\Omega}^{-\sbullet}[u^{-1},u\rrbracket[\hbar^{-1},\hbar\rrbracket[2d]).\end{equation*}
Here the action is induced (through conjugation) by the action on $\WW$ and by the action (by modding out $\hbar$) on $\hat{\Omega}$.
From now on the notation $\mathbb{L}^\sbullet$ will refer to this complex. The differential $\partial_\mathbb{L}$ on $\mathbb{L}^\sbullet$ is given by
viewing it as the usual morphism space internal to chain complexes. In order to derive the equivariant algebraic index theorem we shall have to pair
classes in Lie algebra cohomology with values in $\mathbb{L}^\sbullet$ with periodic cyclic chains of $\AHc\rtimes\Gamma $ using the equivariant Gelfand--Fuks map.
Since the trace on $\AHc\rtimes\Gamma $ is supported at the identity we only need to consider the component of the cyclic complexes supported at the identity.

\begin{definition}[Homogeneous Summand]
	
	\leavevmode
	
 Let $CC^{per}_\sbullet(\AH\rtimes\Gamma)_e$ be the subcomplex spanned (over $\C[u^{-1},u\rrbracket$) by the chains
 \begin{equation*}a_0\gamma_0\otimes\ldots \otimes a_n\gamma_n\hspace{0.5cm}\mbox{such that}\hspace{0.3cm} \gamma_0\gamma_1\ldots\gamma_n=e\in \Gamma,\end{equation*}
 where $e$ denotes the neutral element of $\Gamma$.
\end{definition}

\begin{proposition}
 The map \begin{equation*}D\colon CC^{per}_{\sbullet}(\AH\rtimes\Gamma)\longrightarrow C_\sbullet(\Gamma;CC^{per}_\sbullet(\AH))\end{equation*}
 given by composing the projection \begin{equation*}CC_{\sbullet}^{per}(\AH\rtimes \Gamma)\longrightarrow CC^{per}_{\sbullet}(\AH\rtimes\Gamma)_e\end{equation*} with the
 quasi-isomorphism of theorem \ref{thm:A13} in the appendix is a morphism of complexes.
\end{proposition}
\noindent The proof is contained in the appendix.

As in lemma \ref{O=O}, the canonical inclusion $\iota$
$$
\left(\Omega^\sbullet(M\times_\Gamma E\Gamma)[\hbar^{-1},\hbar\rrbracket[u^{-1},u\rrbracket[2d], d\right)\rightarrow \left(\Omega^\sbullet(M\times_\Gamma E\Gamma;\hat{\Omega}^{-\sbullet}[\hbar^{-1},\hbar\rrbracket[u^{-1},u\rrbracket[2d]), \nabla_{F}+\hat{d}\right)
$$
is a quasi-isomorphism and we will denote its quasi-inverse by $T$
	\begin{equation*}
	T\colon\Omega^\sbullet(M\times_\Gamma E\Gamma;\hat{\Omega}^{-\sbullet}[\hbar^{-1},\hbar\rrbracket[u^{-1},u\rrbracket[2d])\longrightarrow
	\Omega^\sbullet(M\times_\Gamma E\Gamma)[\hbar^{-1},\hbar\rrbracket[u^{-1},u\rrbracket[2d]\end{equation*}

\begin{definition}\label{pairing}
	We define the pairing
\begin{equation*}\langle\cdot,\cdot\rangle\colon \Omega^\sbullet\left(M\times_\Gamma E\Gamma;\LL^\sbullet\right)\times C_\sbullet(\Gamma;CC^{per}_\sbullet(\AHc(M)))
\longrightarrow \C[\hbar^{-1},\hbar\rrbracket[u^{-1},u\rrbracket
\end{equation*}
as follows. In the following let  $\alpha =a \otimes (g_1 \otimes g_2\otimes\ldots\otimes  g_{p})\in CC^{per}_{k-p}(\AH(M))\otimes\left(\C\Gamma\right)^{\otimes p}$   and let
$\phi\in \Omega^\sbullet\left(M\times_\Gamma E\Gamma;\LL^\sbullet\right)$. We define

\begin{equation*}
\langle\phi, \alpha \rangle:= \int \limits_{\Delta^p\times M \times  g_1 \times\ldots \times g_p}T\phi_p(J^\infty_{F_p}(a)),
\end{equation*}
	where $J^\infty_{F_p}$ is the map given by taking the $\infty$-jets of elements of $\AH(M)$ relative to the Fedosov connection $\nabla_{F_p}$ as in the example \ref{JF}  using the section $F_p$ over $M_\Gamma^k$.
Since the integral of $\xi\in\Omega^k\left(M\times_\Gamma E\Gamma\right)$
over any simplex $\Delta^p$ for $p>k$ will vanish, the pairing $\langle\cdot,\cdot\rangle$ extends by linearity to $C_\sbullet(\Gamma;CC^{per}_\sbullet(\AHc(M)))$.
\end{definition}

\begin{lemma} \label{pairinggf}
We have:
\[\langle(\tilde{\nabla}_F+\partial_{\mathbb{L}})\phi, \alpha \rangle = (-1)^{|\phi|+1}\langle\phi, (\delta_\Gamma +b+uB)\alpha \rangle\]
\end{lemma}
\begin{proof}
\begin{equation*}
\langle (\tilde{\nabla}_F+\partial_{\mathbb{L}})\phi, \alpha \rangle= \int \limits_{\Delta^p\times M \times  g_1\times\ldots\times g_p}T
((\tilde{\nabla}_F+\partial_{\mathbb{L}})\phi_p)(J^\infty_{F_p}(a))
\end{equation*}
Notice that $(\partial_{\mathbb{L}}\phi_p)(J^\infty_{F_p}(a)) = \hat{d}((\phi_p)(J^\infty_{F_p}(a))- (-1)^{|\phi|}\phi_p(J^\infty_{F_p}((b+uB)a))$.
Also, since $\tilde{\nabla}_F (J^\infty_{F_p}(a)) =0$ we have $(\tilde{\nabla}_F \phi_p)(J^\infty_{F_p}(a)) =\tilde{\nabla}_F (\phi_p)(J^\infty_{F_p}(a))$. Combining these formulas we obtain that $\langle (\tilde{\nabla}_F+\partial_{\mathbb{L}})\phi, \alpha \rangle$ equals
\begin{multline}\label{eq:prest}
\int \limits_{\Delta^p\times M \times  g_1\times\ldots\times g_p}T \left(
(\tilde{\nabla}_F+ \hat{d})(\phi_p(J^\infty_{F_p}(a)) - (-1)^{|\phi|}\phi_p(J^\infty_{F_p}((b+uB)a))) \right)=\\
\int \limits_{\Delta^p\times M \times  g_1\times\ldots\times g_p}d T (\phi_p(J^\infty_{F_p}(a)) - (-1)^{|\phi|}\langle \phi ,(b+uB)\alpha \rangle
\end{multline}
Applying  Stokes' formula to $\int \limits_{\Delta^p\times M \times  g_1\times\ldots\times g_p}d T (\phi_p(J^\infty_{F_p}(a))$ and noticing that the collection of forms $\{T (\phi_p(J^\infty_{F_p}(a))\}$ is compatible we see that
\begin{equation}\label{eq:post}
\int \limits_{\Delta^p\times M \times  g_1\times\ldots\times g_p}d T (\phi_p(J^\infty_{F_p}(a)) =   (-1)^{|\phi|+1}\langle \phi, \delta_\Gamma \alpha\rangle
\end{equation}
The statement of the lemma now follows from \eqref{eq:prest} and \eqref{eq:post}.

\end{proof}

Recall that we have a cap-product
\[
C_\sbullet(\Gamma;CC^{per}_\sbullet(\AHc(M))) \otimes C^\bullet(\Gamma, \C) \overset{\cap}{\longrightarrow} C_\sbullet(\Gamma;CC^{per}_\sbullet(\AHc(M))).
\]
\begin{definition}\label{C}
	Let $\xi \in C^\bullet(\Gamma, \C)$ be a cocycle. Define
	\begin{equation*}
	I_\xi\colon C^\sbullet_{Lie}(\g,\lsp(2d);\LL^\sbullet)\longrightarrow CC^{\sbullet+|\xi|}_{per}(\AHc\rtimes\Gamma)
	\end{equation*}
	by
\begin{equation*}
I_\xi (\lambda)(a)=\displaystyle  \epsilon(|\lambda|)\langle GF_\Gamma(\lambda),D(a) \cap \xi \rangle
\end{equation*}
for all $\lambda\in C^\sbullet_{Lie}(\g,\lsp(2d);\LL^\sbullet)$ and
	$a\in CC^{per}_\sbullet(\AHc\rtimes\Gamma)$, where
\[
\epsilon(m)=(-1)^{m(m+1)/2}.
\]
\end{definition}

\begin{proposition}\label{rly}
	The map
	\begin{equation*}I_\xi \colon \left(C^\sbullet_{Lie}(\g,\lsp(2d);\LL^\sbullet),\partial_{Lie}+\partial_\LL\right)\longrightarrow
	\left(CC^{\sbullet+|\xi|}_{per}(\AHc\rtimes\Gamma ),(b+uB)^*\right)\end{equation*} is a morphism of complexes.
\end{proposition}

\begin{proof}
	\leavevmode
	
Using Theorem \ref{gfg} and Lemma \ref{pairinggf}  we have
\begin{multline*}
I_\xi  ((\partial_{Lie}+(-1)^r\partial_{\LL})\lambda))(a) = \epsilon(|\lambda|+1)\langle GF_\Gamma((\partial_{Lie}+(-1)^r\partial_{\LL})\lambda)),D(a) \cap \xi \rangle =\\
\epsilon(|\lambda|+1)\langle(\tilde{\nabla}_F+\partial_{\mathbb{L}})GF_\Gamma(\lambda), D(a) \cap \xi \rangle= (-1)^{|\lambda|+1}\epsilon(|\lambda|+1)\langle GF_\Gamma(\lambda), (\delta_\Gamma +b+uB) (D(a)) \cap \xi \rangle=\\
 \epsilon(|\lambda|)\langle GF_\Gamma(\lambda),  (D((b+uB)a)) \cap \xi \rangle =I_\xi  (\lambda)((b+uB)a)
\end{multline*}
and the statement follows.
\end{proof}

\begin{remark} The induced map on cohomology $ I_\xi\colon \CH^\sbullet(\g,\lsp(2d);\mathbb{L}^\sbullet)\longrightarrow \HC^{\sbullet+|\xi|}_{per}(\AHc\rtimes\Gamma)$ is easily seen to depend only on the cohomology class $[\xi]\in \CH^\sbullet(\Gamma, \C)$.
\end{remark}

\section{Evaluation of the equivariant classes}\label{6}

In the previous sections we defined the map
\begin{equation*} I_\xi\colon \CH^0(\g,\lsp(2d);\mathbb{L}^\sbullet)\longrightarrow \HC^{k}_{per}(\AHc\rtimes\Gamma),\end{equation*} where $k= |\xi|$.
The last step in proving the main result of this paper is to evaluate the classes appearing in Lie  algebraic index theorem \ref{UAIT}.

First of all we consider the image under $I_\xi$ of the trace density $\hat{\tau}_a$.
Consider   the map
\begin{equation*}
\langle GF_\Gamma(\hat{\tau}_a), \cdot \rangle \colon C_0(\Gamma;C_0(\AHc))\longrightarrow \C[\hbar^{-1},\hbar\rrbracket.
\end{equation*}
Since in degree $0$  the equivariant Gelfand--Fuks map is given by the ordinary Gelfand--Fuks map on $M$, this map coincides with the canonical trace $Tr$ (cf. the proof of theorem~\ref{thm:AIT}). It follows that
\[
\langle GF_\Gamma(\hat{\tau}_a), \alpha \otimes (\gamma_1\otimes \ldots \gamma_k)  \cap \xi \rangle = \xi (\gamma_1, \ldots, \gamma_k) Tr (\alpha)
\]
From this discussion we obtain the following:
\begin{proposition}\label{prop:trx}
We have $I_\xi (\hat{\tau}_a) = Tr_\xi$ where $Tr_\xi$ is a  cocycle
on $\AHc(M)\rtimes\Gamma$ given by

\begin{equation}
Tr_\xi(a_0\gamma_0\otimes\ldots\otimes a_k\gamma_k)=\xi(\gamma_1,\ldots, \gamma_k)Tr(a_0\gamma_0(a_1)\ldots(\gamma_0\gamma_1\ldots\gamma_{k-1}(a_k))
\end{equation}

if $\gamma_0\gamma_1\ldots\gamma_k=e$ and $0$ otherwise.

\end{proposition}


\begin{definition}
 The equivariant Weyl curvature $\theta_\Gamma$ is defined as the image of $\hat{\theta}$ under $GF_\Gamma$ followed by ($\Ch$-linear extension of)
 the map in Theorem \ref{Borel=Borel}. Similarly, the equivariant $\hat{A}$-genus of $M$, denoted $\hat{A}(M)_\Gamma$, is defined as the image of $\hat{A}$ under the equivariant Gelfand--Fuks map
 followed by ($\Ch$-linear extension of) the isomorphism in Theorem \ref{Borel=Borel}.
\end{definition}

\begin{example}
	Let us provide an example of the characteristic class $\theta_\Gamma$. To do this consider the example of group actions on deformation quantization given in \cite{classif}. Namely, we consider the symplectic manifold $\R^2/\Z^2=\T^2$, the $2$-torus, with the  symplectic structure $\omega=dy\wedge dx$ induced from the standard one on $\R^2$, where $x$, $y \in \R/\Z$ are the standard coordinates on $\T^2$. We then consider the action of $\Z$ on $\T^2$ by symplectomorphisms where the generator of $\Z$ acts by $T\colon (x,y)\mapsto (x+x_0,y+y_0)$. Note that, for a generic pair $(x_0,y_0)$, the quotient space is not Hausdorff.
	
	%
	
%
	%
	
The Fedosov connection $\nabla_F$ given as in Example \ref{loctriv} descends to the connection on $\T^2$ which is, moreover, $\Z$-invariant (where we endow $C^\infty(\T^2, \WW)$ with the action of $\Z$ induced by the symplectic action on $\T^2$).  It follows that $\AH = \Ker \nabla_F$ is a $\Z$-equivariant deformation
with the characteristic class $\frac{\omega}{i\hbar}$.

We can  obtain a more interesting example by modifying the previous one as follows (cf. \cite{classif}). Let $u \in C^\infty(\T^2, \WW)$ be an invertible element such that $u^{-1} (\nabla_F u)$ is central. Define a new action of $\Z$ on $C^\infty(\T^2, \WW)$ where the generator acts by
\[
w \mapsto u^{-1}(T w) u.
\]
$\Ker \nabla_F$ is again invariant under this action and we thus obtain an action of $\Z$ on $\AH$.

To describe its characteristic class note that, since $E \Z \cong \R$, we find that the cohomology
$\CH^\sbullet_\Z(\T^2) =  \CH^\sbullet(\R \times_\Z \T^2 ) \cong \CH^\sbullet(\T^3)$.
Let $\nu$ be a  compactly supported $1$-form on $\R$ with  $\int_\R \nu =1$. Denote by $\tau$
 the translation $t\rightarrow t-1$. Then
\[
 \tilde{\alpha}=\sum_{n\in \Z}(\tau^*)^n(\nu)\wedge (T^*)^n (U^{-1}\nabla_FU)
\]
 is a $\Z$-invariant form on $\R\times\T^2$, hence a lift of a form, say $\alpha$, on $\R\times_{\Z}\T^2=\T^3$. The characteristic class of the associated $\Z$-equivariant deformation is  equal to
 \begin{equation*}
 \theta_\Z=\frac{\omega}{i\hbar}+\alpha.
 \end{equation*}
\end{example}

Finally we arrive at the main theorem of this paper.
Let $\mathcal{R} \colon H^{even}_\Gamma(M) \to H^\sbullet_\Gamma(M)[u]$ be given by
\[\mathcal{R}(a) = u^{\deg a/2} a\]
and recall the morphism defined in \eqref{eq:Phi}
\begin{equation*}
\Phi\colon H^\sbullet_\Gamma(M)\longrightarrow HC^\sbullet_{per}(C_c^\infty(M)\rtimes\Gamma).
\end{equation*}
\begin{theorem}[Equivariant Algebraic Index Theorem]\label{EAIT}

\leavevmode

 Suppose $a\in CC_0^{per}(\AHc\rtimes\Gamma )$ is a cycle, then we have
 \begin{equation*}Tr_\xi(a)=\left\langle \Phi\left(\mathcal{R}\left(\hat{A}(M)_\Gamma e^{\theta_\Gamma}\right)[\xi]\right), \sigma(a)\right\rangle\end{equation*}
 where $\langle\cdot,\cdot\rangle$ denotes the pairing of $CC^\sbullet_{per}$ and $CC_\sbullet^{per}$.
\end{theorem}
\begin{proof}
The theorem follows from Theorem \ref{UAIT} by applying the morphism $I_\xi$. The image of $\tau_a$ under $I_\xi$ is $Tr_\xi$ (cf. Proposition \ref{prop:trx}). On the other hand, by equation \eqref{eq:Phi},
\[
\left[I_\xi\left(\sum_{p\geq 0}\left(\hat{A}_fe^{\hat{\theta}}\right)_{2p}u^p\hat{\tau}_t\right)\right] = \Phi\left(\mathcal{R}\left(\hat{A}(M)_\Gamma e^{\theta_\Gamma}\right)[\xi]\right).
\]
\end{proof}
Note that the form of the theorem \ref{Mainresult} stated in the introduction follows by considering the pairing of periodic cyclic cohomology and $K$-theory using the Chern--Connes character \cite{Loday}.


\appendix
\appendixpage

Below we shall fix our conventions with regard to cyclic/simplicial structures and homologies.
We will also define the complexes we use to describe group (co)homology, Lie algebra cohomology and cyclic (co)homology. The general reference for this section is \cite{Loday}.

Fix a field $k$ of characteristic $0$.

\section{Cyclic/simplicial structure}

Let $\Lambda$ denote the cyclic category. Instead of giving the intuitive definition let us simply give a particularly useful presentation. The cyclic category
$\Lambda$ has objects $[n]$ for each $n\in\Z_{\geq 0}$ and is generated by

\begin{align*}
 \delta_i^n&\in\Hom([n-1],[n])&\mbox{and}&\hspace{0.3cm}\sigma_i^n\in\Hom([n+1],[n]) &\hspace{0.5cm}\mbox{for}\hspace{0.2cm} 0\leq i\leq n\\
 t_n&\in\Hom([n],[n])& & & \hspace{0.3cm}\mbox{for all}\hspace{0.3cm}n\in\Z_{\geq0}
\end{align*}

 with the relations
 \begin{align*}
  \delta_j^n\circ\delta_i^{n-1}&=\delta_i^n\circ\delta_{j-1}^{n-1}&\text{if  }& i<j&\sigma_j^n\circ\sigma_i^{n+1}&= \sigma^n_i\circ\sigma^{n+1}_{j+1}&\text{if  }&i\leq j&\\
  \sigma_j^n\circ\delta_i^{n+1}&=\delta_i^n\circ\sigma_{j-1}^{n-1}&\text{if  }&i<j&\sigma_j^n\circ\delta_i^{n+1}&=\Id_{[n]}&\text{if  }&i=j,j+1&\\
  \sigma_j^n\circ\delta_i^{n+1}&=\delta^n_{i-1}\circ\sigma_j^{n-1}&\text{if  }&i>j+1&t_n^{n+1}&=\Id_{[n]}&\\
  t_n\circ\delta_i^n&=\delta_{i+1}^n\circ t_{n-1}&\text{if  }&0\leq i< n&t_n\circ\delta^n_n&= \delta_0^n&\\
  t_n\circ \sigma^n_i&=\sigma_{i+1}^n\circ t_{n+1}&\text{if  }& 0\leq i<n&t_n\circ\sigma_n^n&=\sigma_0^n\circ t_{n+1}\circ t_{n+1}.&\\
 \end{align*}

Using only the generators $\delta_i^n$ and $\sigma_i^n$ and relations not involving $t_n$'s gives a presentation of the simplicial category $\triangle$. A contravariant
functor from $\Lambda$ ($\triangle$) to the category of $k$-modules is called a cyclic (simplicial) $k$-module.

\begin{definition}
Given a unital associative $k$-algebra $A$ we shall denote by $A^\natural$ the functor $\Lambda^{op}\rightarrow k-Mod$ given by
$A^\natural([n])=A^{\otimes n+1}$ and
\begin{align*}
 \delta^n_i(a_0\otimes\ldots\otimes a_n)=&a_0\otimes \ldots\otimes a_ia_{i+1}\otimes\ldots\otimes a_n&\hspace{0.3cm}\mbox{if}\hspace{0.3cm} 0\leq i<n \\
 \delta^n_n(a_0\otimes\ldots\otimes a_n)=&a_na_0\otimes a_1\otimes\ldots\otimes a_{n-1}& \\
 \sigma_i^n(a_0\otimes\ldots\otimes a_n)=&a_0\otimes\ldots\otimes a_i\otimes 1\otimes a_{i+1}\otimes \ldots\otimes a_n&\hspace{0.3cm}\mbox{for all}\hspace{0.3cm}0\leq i\leq n\\
 t_n(a_0\otimes\ldots\otimes a_n)=&a_1\otimes\ldots\otimes a_n\otimes a_0& \\
\end{align*}
Note that if $A$ admits a group action of the group $G$ by unital algebra homomorphisms then $G$ also acts on $A^\natural$ (diagonally).
\end{definition}

\begin{definition}
Given a group $G$ we shall denote by $G^{k\natural}$ the functor $\Lambda^{op}\rightarrow k-Mod$ given by
$G^{k\natural}([n])=(kG)^{\otimes n+1}$ and
\begin{align*}
 \delta_i^n(g_0\otimes \ldots\otimes g_n)=&g_0\otimes \ldots\otimes \hat{g_i}\otimes \ldots\otimes g_n & \hspace{0.3cm}\mbox{for all} \hspace{0.3cm} 0\leq i\leq n \\
 \sigma_i^n(g_0\otimes \ldots\otimes g_n)=&g_0\otimes\ldots\otimes g_i\otimes g_i\otimes g_{i+1}\otimes\ldots\otimes g_n&\hspace{0.3cm}\mbox{for all}\hspace{0.3cm} 0\leq i\leq n\\
 t_n(g_0\otimes \ldots\otimes g_n)=&g_1\otimes g_2\otimes\ldots\otimes g_n\otimes g_0.& \\
\end{align*}
Note that $G$ acts on $G^{k\natural}$ from the right by
$g\cdot(g_0\otimes \ldots\otimes g_n)=g^{-1}g_0\otimes \ldots\otimes g^{-1}g_n$.
\end{definition}

\begin{definition}
Given two cyclic $k$-modules $A^\natural$ and $B^\natural$ we shall denote by $A\natural B$ the cyclic $k$-module given by
$A\natural B([n])=A^\natural([n])\otimes B^\natural([n])$ with the diagonal cyclic structure.
\end{definition}

\subsection{Cyclic homologies}

Given a cyclic $k$-module $M^\natural$ we can consider four different complexes associated to the simplicial/cyclic structure. To define them we shall
first define two operators: $b$ and $B$.

The first is induced through the Dold--Kan correspondence and uses only the simplicial structure. It is given by
\begin{equation*}b_n=\sum_{i=0}^n(-1)^i\delta_i^n\colon M^\natural([n])\longrightarrow M^\natural([n-1]).\end{equation*} By using the simplicial identities above
it is easily verified that $b_{n-1}b_n=0$. To define the ``Hochschild'' complex it is enough to have just the operators $b_n$.

To define the three cyclic complexes we shall use the operator
\begin{equation*}B_n=(t_{n+1}^{-1}+(-1)^n)\circ\sigma^n_n\circ\left(\sum_{i=0}^n(-1)^{in}t_n^i\right)\colon M^\natural([n])\longrightarrow M^\natural([n+1]).\end{equation*}
Note that $B_{n+1}B_n=0$ since
\begin{equation}\label{BB=0}\sum_{i=0}^{n+1}(-1)^{i(n+1)}t^i_{n+1}\circ(t_{n+1}^{-1}+(-1)^n)=\sum_{i=0}^{n+1}(-1)^{i(n+1)}(t_{n+1}^{i-1}+(-1)^nt_{n+1}^i)=0.\end{equation}
Vanishing of the above expression follows since the sum telescopes except for the first term $t_{n+1}^{-1}$ and the last term $-(-1)^{n(n+1)}t_{n+1}^{n+1}$, which also
cancel each other. Note also that
\begin{equation}\label{bB=Bb}b_{n+1}B_n+ B_{n-1}b_n=0,\end{equation}
this can be seen by writing out both operators as sums of operators in the normal form $\delta_k^n\sigma_l^{n-1}t_n^i$.

From now on we will drop the subscripts of the $b$ and $B$ operators. The cyclic module $M^\natural$ gives rise to a graded module $\{M^\natural_n\}_{n\in\Z_{\geq 0}}$ by
$M^\natural_n=M^\natural([n])$. Then we see that the operator $b$ turns $M^\natural$ into a chain complex.
\begin{definition}
 The \emph{Hochschild complex} $(C^{Hoch}_\bullet(M^\natural),b)$ of the cyclic module $M^\natural$ is defined as
 $C^{Hoch}_n(M^\natural):=M^\natural([n])$ equipped with the boundary operator $b$ (of degree $-1$). The corresponding homology shall be denoted
 $\CH\CH_\bullet(M^\natural)$.
\end{definition}

Note that we have not used the full cyclic structure of $M^\natural$ to construct the Hochschild complex. In fact one can form the Hochschild complex
$\left(C_\sbullet(M^\triangle),b\right)$ of any simplicial $k$-module $M^{\triangle}$ in exactly the same way.

Note that by \eqref{bB=Bb} and \eqref{BB=0} we find that $(b+B)^2=0$. This  implies that we could consider a certain double complex with columns given by the Hochschild complex. Note
however that, if $b$ is of degree $-1$ on the Hochschild complex, the operator $B$ is naturally of degree $+1$. We can consider a new grading for which the operator
$b+B$ is of homogeneous degree $-1$. In order to make this grading easy to see, it will be useful to introduce the formal variable $u$ of degree $-2$.
This leads us to several choices of double complexes.

\begin{definition}
 We define the \emph{cyclic} complex by
\begin{equation*}(CC_\bullet(M^\natural),\delta^\natural):=\left(\bigslant{C_\bullet^{Hoch}(M^\natural)[u^{-1},u\rrbracket}{C_\bullet^{Hoch}(M^\natural)\llbracket u\rrbracket}, b+uB\right),\end{equation*}
the \emph{negative cyclic} complex by
\begin{equation*}(CC^{-}_\bullet(M^\natural),\delta^\natural_{-}):=\left(C_\bullet^{Hoch}(M^\natural)\llbracket u\rrbracket,b+uB\right)\end{equation*}
and finally the \emph{periodic cyclic} complex by
\begin{equation*}(CC^{per}_\bullet(M^\natural),\delta^\natural_{per}):=\left(C_\bullet^{Hoch}(M^\natural)[u^{-1}, u\rrbracket,b+uB\right).\end{equation*}
Here $u$ denotes a formal variable of degree $-2$.
 The corresponding homologies will be denoted $\HC_\bullet(M^\natural)$, $\HC^{-}_\bullet(M^\natural)$ and $\HC^{per}_\bullet(M^\natural)$ respectively.
 The  cyclic cochain complexes $CC_{per}^\sbullet(M^\natural)$, $CC_{-}^\sbullet(M^\natural)$ and $CC^{\sbullet}(M^\natural)$ are defined as the $k$-duals
 of the chain complexes.
\end{definition}

We shall often omit the superscripts $\natural$ when there can be no confusion as to what the cyclic structures are.

\begin{remark}\label{HHtoCH}
Note that every ``flavor'' of cyclic homology comes equipped with spectral sequences induced from the fact that they are realized as totalizations of
a double complex. The double complex corresponding to cyclic homology is bounded (second octant) and therefore the spectral sequence which starts by taking homology on
columns converges to $\HC_\bullet$. The negative (or periodic) cyclic double complex is unbounded, but concentrated in the (second,) third, fourth and fifth  octant. This
means that the spectral sequence starting with taking homology in the columns converges again to $\HC^-_\bullet$ (or $\HC^{per}_\bullet$). Note however that in this case
the negative (or periodic) cyclic homology is given by the product totalization.
\end{remark}

The remark \ref{HHtoCH} provides the proof of the following proposition.

\begin{proposition}\label{Hochsuf}
 Suppose $M^\natural$ and $N^\natural$ are two cyclic $k$-modules and $\phi\colon N^\natural\longrightarrow M^\natural$ is a map of cyclic modules that induces
 an isomorphism on Hochschild homologies. Then $\phi$ induces an isomorphism on cyclic, negative cyclic and periodic cyclic homologies as well.
\end{proposition}
\begin{proof}

 \leavevmode

 \noindent The proof follows since $\phi$ induces isomorphisms on the first pages of the relevant spectral sequences, which converge.
\end{proof}

\subsection{Replacements for cyclic complexes}
It will often be useful to consider different complexes that compute the various cyclic homologies. We shall give definitions of the complexes that are used in the
main body of the article here.

\subsubsection{Crossed product}
Suppose $A$ is a unital $k$-algebra and $G$ is a group acting on the left by unital algebra homomorphisms. We denote by $A\rtimes G$ the crossed product algebra
given by $A\otimes kG$ as a $k$-vector space and by the multiplication rule $(ag)(bh)=ag(b)gh$ for all $a,b\in A$ and $g,h\in G$. Note that the cyclic structure of
$(A\rtimes G)^\natural$ splits over the conjugacy classes of $G$. Namely, given a tensor $a_0g_0\otimes a_1g_1\otimes \ldots\otimes a_ng_n$, the conjugacy class of the
product $g_0\cdot\ldots\cdot g_n$ is invariant under $\delta^n_i$, $\sigma^n_i$ and $t_n$ for all $i$ and $n$. So we have
\begin{equation*}(A\rtimes G)^\natural=\bigoplus_{x\in\langle G\rangle}(A\rtimes G)^\natural_x\end{equation*} where we denote the set of conjugacy classes of $G$ by $\langle G\rangle$ and the span of all tensors $a_0g_0\otimes\ldots\otimes a_ng_n$ such that $g_0\cdot\ldots\cdot g_n\in x$ by
$(A\rtimes G)^\natural_x$.
The summand $(A\rtimes G)^\natural_e$, here $e=\{e\}$ the conjugacy class of the neutral element, is called the homogeneous summand.

We shall use the specialized notation $A\natural G:=A^\natural\natural G^{k\natural}$. Note that $A\natural G$ carries a right $G$ action given by the diagonal action (the left action on $A$ is converted to a right action by inversion, i.e.
$G\simeq G^{op}$). Thus the co-invariants $(A\natural G)_G=\bigslant{A\natural G}{\langle a-g(a)\rangle}$ form another cyclic $k$-module.

\begin{proposition}\label{above3}
 The homogeneous summand of $(A\rtimes G)^\natural$ is isomorphic to the co-invariants of $A\natural G$.
 \begin{equation*}(A\rtimes G)^\natural_e\stackrel{\sim}{\longrightarrow} (A\natural G)_G.\end{equation*}
\end{proposition}
\begin{proof}
 \leavevmode

 \noindent Consider the map given by
 \begin{equation*}a_0g_0\otimes\ldots\otimes a_ng_n\mapsto (g^{-1}_0(a_0)\otimes a_1\otimes g_1(a_2)\otimes \ldots\otimes g_1\ldots g_{n-1}(a_n))\natural
  (e\otimes g_1\otimes g_1g_2\otimes \ldots\otimes g_1\cdot\ldots\cdot g_n),
 \end{equation*}
 it is easily checked to commute with the cyclic structure and allows the inverse given by
 \begin{equation*}(a_0\otimes\ldots\otimes a_n)\natural(g_0\otimes \ldots\otimes g_n)\mapsto g_n^{-1}(a_0)g_n^{-1}g_0\otimes g_0^{-1}(a_1)g_0^{-1}g_1\otimes \ldots
  \otimes g_{n-1}^{-1}(a_n)g_{n-1}^{-1}g_n
 \end{equation*}
this last tensor can also be expressed as $g_n^{-1}a_0g_0\otimes g_0^{-1}a_1g_1\otimes \ldots \otimes g_{n-1}^{-1}a_ng_n$.
\end{proof}
\begin{definition}
 Suppose $(M_\bullet,\partial)$ is a right $kG$-chain complex. Then we define the group homology of $G$ with values in $M$ as
 \begin{equation*}(C_\bullet(G;M),\delta_{(G,M)}):=\Tot^{\scriptscriptstyle\prod}M_\bullet\otimes_{kG}C_\bullet^{Hoch}(G)\end{equation*}
 where we consider the tensor product of $kG$-chain complexes with the obvious structure of left $kG$-chain complex on $C_\bullet^{Hoch}(G)$. Note that
 this means that
 \begin{equation*}C_n(G;M)=\prod_{p+q=n}M_p\otimes_{kG}C^{Hoch}_q(G)\end{equation*} and
 \begin{equation*}\delta_{(G,M)}=\partial\otimes\Id+\Id\otimes b\end{equation*} where we use the Koszul sign convention.

\end{definition}

\begin{proposition}\label{above2}
 Suppose $M$ is a right $kG$-module. Then $M\otimes kG$ with the diagonal right action is a free $kG$-module.
\end{proposition}
\begin{proof}

 \leavevmode
 \noindent Let us denote the $k$-module underlying $M$  by $F(M)$, then $F(M)\otimes kG$ denotes the free (right) $kG$-module induced by the $k$-module underlying $M$.
 Consider the map
 \begin{equation*}M\otimes kG\longrightarrow F(M)\otimes kG\end{equation*} given by $m\otimes g\mapsto mg^{-1}\otimes g$. It is obviously a map of $kG$-modules and allows for the inverse
 $m\otimes g\mapsto mg\otimes g$.
\end{proof}

\begin{proposition}\label{above1}
 Suppose $F$ is a free right $kG$-module (we view it as a chain complex concentrated in degree $0$ with trivial differential) then there exists a contracting homotopy
 \begin{equation*}H_F\colon C_\bullet(G;F)\longrightarrow C_{\bullet +1}(G;F).\end{equation*}
 Suppose $(F_\bullet,\partial)$ is a quasi-free right $kG$-chain complex (i.e. $F_n$ is a free $kG$-module for all $n$) then the homotopies $H_{F_n}$ give rise to a
 quasi-isomorphism \begin{equation*}\left((F_\bullet)_G, \partial\right)\stackrel{\sim}{\longrightarrow} (C_\bullet(G;F),\delta_{(G,F)}).\end{equation*}
 \end{proposition}
 \begin{proof}

  \leavevmode

  \noindent Note that $F\simeq M\otimes kG$ since it is a free module. So we find that
  \begin{equation*}C_p(G;F)=(M\otimes kG)\otimes_{kG}(kG)^{\otimes p+1}\simeq M\otimes(kG)^{\otimes p+1}\end{equation*} by the map
  $m\otimes g\otimes g_0\otimes \ldots\otimes g_p\mapsto m\otimes gg_0\otimes \ldots \otimes gg_p.$ Using this normalization we consider the map
  $H_M$ given by
  \begin{equation*}m\otimes g_0\otimes\ldots\otimes g_p\mapsto m\otimes e\otimes g_0\otimes\ldots \otimes g_p\end{equation*} and note that indeed
  \begin{equation*}\delta_G^{p+1}H_M+H_M\delta_G^{p}=\Id\end{equation*} (we denote $\delta_G:=\delta_{(G,M)}=\Id\otimes b$) for all $p>0$.

  Now for the second statement we find that $F_n\simeq M_n\otimes kG$ for each $n$ since
  it is quasi-free. For each $n$ we have the homotopy $H_n:=H_{F_n}$ given by the formula above on $C_\bullet(G;F_n)$.
  Then we consider the map
 \begin{equation*}Q_H\colon (F_p)_G\longrightarrow C_p(G;F)\end{equation*} given by
 \begin{equation*}Q_F([f])=f-\delta^1_GHf+\sum_{q=1}^\infty(-H\partial)^qf-\partial(-H\partial)^{q-1}Hf-\delta_G^{q+1}(-H\partial)^qHf\end{equation*}
 where we have dropped the subscript from $H$ and we denote the class of $f$ in the co-invariants $F_G$ by $[f]$ . One may check
 by straightforward computation that $Q_F$ is a well-defined morphism of complexes. Now we note that the double complex defining $C_\bullet(G;F)$ is concentrated in the
 upper half plane and therefore comes with a spectral sequence with first page given by $H_p(G;F_q)$ which converges to $\CH(C_{p+q}(G;F))$ (group homology). Note
 however that since $F_\bullet$ is quasi-free we find that $H_p(G,F_q)=0$ unless $p=0$ and $H_0(G,F_q)=(F_q)_G$. Thus, since $Q_F$ induces an isomorphism on the first
 page and the spectral sequence converges, we find that $Q_F$ is a quasi-isomorphism.
 \end{proof}

As a $kG$-module we see that $A\natural G([n])=A^\natural([n])\otimes G^{k\natural}([n])=B([n])\otimes kG$ with the diagonal action, where
$B([n])=A^{\otimes n+1}\otimes kG^{\otimes n}$. So by proposition \ref{above2} we find that the Hochschild and various cyclic chain complexes corresponding
to $A\natural G$ are quasi-free. Thus we can
construct the quasi-isomorphisms from proposition \ref{above1} for each chain complex associated to the cyclic module $A\natural G$. So we find four quasi-isomorphisms
which we shall denote $Q^{Hoch}$, $Q$, $Q^-$ and $Q^{per}$ corresponding to the Hochschild, cyclic, negative cyclic and periodic cyclic complexes respectively.

\begin{proposition}\label{above4}
 The map \begin{equation*}A\natural G\longrightarrow A^\natural \end{equation*} given by
 \begin{equation*}(a_0\otimes \ldots\otimes a_n)\natural(g_0\otimes \ldots\otimes g_n)\mapsto a_0\otimes \ldots\otimes a_n\end{equation*}
 induces a quasi-isomorphism on all associated complexes.
\end{proposition}
\begin{proof}

 \leavevmode

 \noindent Note that, by proposition \ref{Hochsuf}, it is sufficient to prove the statement for the Hochschild complexes.
 Let us denote the standard free resolution of $G$ by $F(G)$, note that \begin{equation*}F(G)=(C^{Hoch}_\bullet(G^{k\natural}),b).\end{equation*}
 The map given above is obtained by first applying the Alexander--Whitney map
 \begin{equation*}C^{Hoch}_n(A^\natural)\otimes C^{Hoch}_n(G^{k\natural})\longrightarrow \bigoplus_{p+q=n}C^{Hoch}_p(A^\natural)\otimes C^{Hoch}_q(G^{k\natural}),\end{equation*}
 which yields a quasi-isomorphism
 \begin{equation*}C^{Hoch}_\bullet(A\natural G)\stackrel{\sim}{\longrightarrow}C^{Hoch}_\bullet(A^\natural)\otimes C^{Hoch}_\bullet(G^{k\natural}),\end{equation*}
 where we consider the tensor product of chain complexes on the right-hand side. Then one simply takes the cap product
 with the generator in $H^*(F(G)^*)\simeq k$, which is also a quasi-isomorphism.  So we find that the map is a quasi-isomorphism for the Hochschild complexes.
\end{proof}

Note that the map given in proposition \ref{above4} is also $G$-equivariant and therefore it induces a map
\begin{equation*}C_\bullet(G;A\natural G)\longrightarrow C_\bullet(G;A^\natural)\end{equation*} which is a quasi-isomorphism when we consider the group homology complex with values in the various
complexes associated to $A^\natural$.

\begin{theorem}\label{thm:A13}
 The composite maps from the Hochschild and various cyclic complexes associated to $(A\rtimes\Gamma)^\natural_e$ to the
 group homology with values in the various Hochschild and cyclic complexes associated to $A^\natural$
 implied by propositions \ref{above3} and \ref{above4}
 are quasi-isomorphisms, i.e. there are quasi-isomorphisms
 \begin{equation*}\left(C^{Hoch}_\bullet\left((A\rtimes G)^\natural_e\right),b\right)\stackrel{\sim}{\longrightarrow}
 C_\bullet(G;C^{Hoch}_\bullet(A))\end{equation*}
 \begin{equation*}\left(CC_\bullet\left((A\rtimes G)^\natural_e\right),\delta^{\natural}\right)\stackrel{\sim}{\longrightarrow}
 C_\bullet(G;CC_\bullet(A))\end{equation*}
 \begin{equation*}\left(CC^{-}_\bullet\left((A\rtimes G)^\natural_e\right),\delta_{-}^{\natural}\right)\stackrel{\sim}{\longrightarrow}
 C_\bullet(G;CC^{-}_\bullet(A))\end{equation*}
 and
 \begin{equation*}\left(CC^{per}_\bullet\left((A\rtimes G)^\natural_e\right),\delta_{per}^{\natural}\right)\stackrel{\sim}{\longrightarrow}
 C_\bullet(G;CC^{per}_\bullet(A)).\end{equation*}
\end{theorem}

\begin{remark}
 Note that since the cyclic and Hochschild complexes are bounded below the product totalizations in our definition of group homology  agrees with the (usual) direct sum
 totalizations. In the periodic cyclic and negative cyclic cases they do not agree in general.
\end{remark}
\begin{remark}
Suppose that a discrete group $\Gamma$ acts on a smooth manifold $M$ by  diffeomorphisms. The above produces a morphism of complexes
$$
CC_\bullet^{per} (C^\infty(M)_c\rtimes \Gamma)\rightarrow C_\bullet (\Gamma,CC_\bullet^{per} (C^\infty_c(M))
$$
Composing it with the   morphism
$$
CC_\bullet^{per} (C^\infty_c(M))\longrightarrow \Omega^\bullet_c(M)[u^{-1},u\rrbracket,
$$
induced by the map
$$
f_0\otimes f_1\otimes\ldots\otimes f_n\mapsto \frac{1}{n!} f_0df_1\ldots df_n
$$
we get a morphism of complexes
\begin{equation}
\label{eq:cyclic homology of crossed product}
CC_\bullet^{per} (C^\infty(M)_c\rtimes \Gamma)\rightarrow C_\bullet (\Gamma,\Omega^\bullet_c(M)[u^{-1},u\rrbracket).
\end{equation}
In the case when $M$ is oriented and the elements of $\Gamma$ preserve orientation, the transpose of this map can be interpreted as a morphism of complexes
\begin{equation}
\label{eq:Phi}
\Phi\colon C^\bullet (\Gamma,\Omega^{\dim(M)-\bullet}(M)[u^{-1},u\rrbracket)\longrightarrow CC^\bullet_{per}(C^\infty_c(M)\rtimes \Gamma),
\end{equation}
compare \cite{NCG} section 3.2.$\delta$.
\end{remark}
\subsubsection{Group Homology}

It is often useful to consider instead of the above complex for group homology an \emph{isomorphic} complex, which we will call the \emph{non-homogeneous complex}.

\begin{definition}
 Suppose $(M_\sbullet,\partial)$ is a right $kG$-chain complex, then we set
 \begin{equation*}\tilde{C}_n(G;M):=\prod_{p+q=n}M_q\otimes (kG)^{\otimes p}.\end{equation*} We define the operators
$\delta_i^p\colon M_\sbullet\otimes (kG)^{\otimes p}\rightarrow M_\sbullet\otimes (kG)^{\otimes p-1}$ by
\begin{equation*}\delta_0^p(m\otimes g_1\otimes\ldots\otimes g_p):=g_1(m)\otimes g_2\otimes\ldots\otimes g_p\end{equation*}
\begin{equation*}\delta_i^p(m\otimes g_1\otimes\ldots\otimes g_p):=m\otimes g_1\otimes\ldots\otimes g_ig_{i+1}\otimes \ldots\otimes g_p\end{equation*}
for all $0<i<p$ and finally
\begin{equation*}\delta_p^p(m\otimes g_1\otimes\ldots\otimes g_p):=m\otimes g_1\otimes\ldots\otimes g_{p-1}.\end{equation*} We define
$(\tilde{C}_\sbullet(G;M), \tilde{\delta}_{(G,M)})$ to be the chain complex given by
\begin{equation*}\tilde{\delta}_{(G,M)}=\partial\otimes\Id+\Id\otimes \delta_G\end{equation*}
where $\delta^p_G=\sum_{i=0}^p\delta_i^p$.
\end{definition}

\begin{proposition}
 There is an isomorphism of chain complexes
 \begin{equation*}C_\sbullet(G;M)\longrightarrow \tilde{C}_\sbullet(G;M).\end{equation*}
\end{proposition}
\begin{proof}

 \leavevmode

 \noindent Consider the map
\begin{equation*}C_n(G;M)\longrightarrow \tilde{C}_n(G;M),\end{equation*}
given by \begin{equation*}m\otimes g_0\otimes\ldots\otimes g_p\mapsto g_0(m)\otimes g_0^{-1}g_1\otimes g_1^{-1}g_2\otimes\ldots\otimes
g^{-1}_{p-1}g_p.\end{equation*}
Note that it commutes with the differentials and allows for the inverse given by
\begin{equation*}m\otimes g_1\otimes\ldots\otimes g_p\mapsto m\otimes e\otimes g_1\otimes g_1g_2\otimes\ldots\otimes g_1\cdot\ldots\cdot g_p.\end{equation*}
\end{proof}

We will usually use this chain complex when dealing with group homology and thus we will drop the tilde in the main body of this article.

\subsection{Lie algebra cohomology}

In this section let us describe the Lie algebra cohomology. Although there are various deep relations between Lie algebra cohomology and cyclic and Hochschild homologies
we have chosen to present the complex in a separate manner. One could compute the Lie algebra cohomology using a Hochschild complex, however in the relative case
(which we need in this article) there are several subtleties that we would rather avoid by considering a different complex.

\begin{definition}
Suppose $\g$ is a Lie algebra over $k$, $\mathfrak{h}\hookrightarrow \mathfrak{g}$ a subalgebra and $(M_\sbullet,\partial)$ is a $\g$-chain complex. Then we denote
\begin{equation*}C^p_{Lie}(\g,\mathfrak{h};M_q):=\Hom_{\mathfrak{h}}\left(\largewedge{p}\g/\mathfrak{h},M_q\right).\end{equation*}
We define operators \begin{equation*}\partial_{Lie}^p\colon C^p_{Lie}(\g,\mathfrak{h};M_q)\longrightarrow C^{p+1}_{Lie}(\g,\mathfrak{h};M_q)\end{equation*}
by
\begin{equation*}\partial_{Lie}^p\phi(X_0,\ldots,X_p)=\sum_{i=0}^p(-1)^iX_i\phi(X_0,\ldots,\hat{X_i},\ldots,X_p)
 \end{equation*}\begin{equation*}+\sum_{0\leq i<j\leq p}(-1)^{i+j}\phi([X_i,X_j],X_0,\ldots, \hat{X_i},\ldots, \hat{X_j},\ldots, X_p)
\end{equation*}
where the hats signify omission. Note that $\partial_{Lie}^{p+1}\partial_{Lie}^p=0$ and $\partial_{Lie}$ commutes with $\partial$ by assumption. Thus we can consider
the totalization of the corresponding double complex. We will denote the corresponding hypercohomology by $\mathbb{H}_{Lie}^\sbullet(\g,\mathfrak{h};M)$.

\end{definition}

\begin{remark}
 We actually only consider the Lie algebra cohomology of infinite dimensional Lie algebras here. For these the complex above is not very useful. The Lie algebras we
 consider come with a topology (induced by filtration) however and so do the coefficients. Using this fact we can consider in the above not simply anti-symmetric linear
 maps, but continuous anti-symmetric linear maps from the completed tensor products.
\end{remark}


\begin{thebibliography}{20}

\bibitem{A-S}
Michael Atiyah and Isadore Singer,
\emph{The Index Of Elliptic Operators on Compact Manifolds},
Bull. Amer. Math. Soc., 69, 422-433, 1963.

\bibitem{RRI}
Paul Bressler, Ryszard Nest and Boris Tsygan,
\emph{Riemann-Roch Theorems Via Deformation Quantization I}
Adv. Math. 167, 1--25, 2002.

\bibitem{RRII}
Paul Bressler, Ryszard Nest and Boris Tsygan,
\emph{Riemann-Roch Theorems Via Deformation Quantization II}
Adv. Math. 167, 26--73, 2002.

\bibitem{NCDG}
Alain Connes,
\emph{Non-Commutative Differential Geometry},
IHES Publ. Math., 62, 257--360, 1985.

\bibitem{NCG}
Alain Connes,
\emph{Non-commutative Geometry},
Academic Press, San Diego, 1990.

\bibitem{C-M}
Alain Connes and Henri Moscovici,
\emph{Hopf Algebras, Cyclic Cohomology and the Transverse Index Theorem}
Commun. Math. Phys., 198, 199-246, 1998.



\bibitem{Globalana}
Paul Bressler, Alexander Gorokhovsky, Ryszard Nest and Boris Tsygan,
\emph{Algebraic Index Theorem for Symplectic Deformations of Gerbes},
Contemporary Mathematics vol. 546 (Non-commutative Geometry and Global analysis), 2011.


\bibitem{Dupont}
Johan Dupont,
\emph{Curvature and Characteristic Classes},
Springer, LNM 640, Heidelberg, 1978.

\bibitem{Fed}
Boris Fedosov,
\emph{The Index Theorem for Deformation Quantization}, in Boundary Value Problems,
Schr\"odinger Operators, Deformation Quantization; Akademie, Advances in Partial
Differential Equations, 319-333, Berlin, 1995.

 \bibitem{Fb}
Boris Fedosov,
\emph{Deformation Quantization and Index Theory},
Akademie Verlag, Berlin, 1st Edition, 1996, chapter 5 and 6.



\bibitem{Nice}
Israel Gelfand,
\emph{Cohomology of Infinite-Dimensional Lie Algebras. Some Questions in Integral Geometry},
Lecture given at ICM, Nice, 1970.

\bibitem{GK}
Israel Gelfand and David Kazhdan,
\emph{Certain Questions of Differential Geometry and the Computation of the Cohomologies of the Lie Algebras of Vector Fields},
Soviet Math. Doklady, 12, 1367-1370, 1971.

\bibitem{Gutsurv}
Simone Gutt,
\emph{Deformation Quantization: an Introduction},
3rd edition, HAL, Monastir, 2005.

\bibitem{classif}
Niek de Kleijn
\emph{Extension and Classification of Group Actions on Formal Deformation Quantizations of Symplectic Manifolds},
preprint, arXiv:1601.05048, 2016.

\bibitem{Loday}
Jean-Louis Loday,
\emph{Cyclic Homology}, 2nd edition,
Springer, GMW 301, Heidelberg, 1998.


\bibitem{AIT}
Ryszard Nest and Boris Tsygan,
\emph{Algebraic Index Theorem}
Commun. Math. Phys. 172, 223--262, 1995.

\bibitem{N-T2}
Ryszard Nest and Boris Tsygan,
\emph{Formal Versus Analytic Index Theorems},
Internat. Math. Res. Notices, 11, 1996.

\bibitem{N-T99}
Ryszard Nest and Boris Tsygan,
\emph{Deformations of Symplectic Lie Algebroids, Deformations of Holomorphic Symplectic Structures, and Index Theorems},
Asian J. Math. 5 (2001), no. 4, 599--635.

\bibitem{Perrot}
Denis Perrot, Rudy Rodsphon, \emph{An equivariant index theorem for hypo-elliptic operators}
arXiv:1412.5042

 


\bibitem{Post2}
Marcus Pflaum, Hessel Posthuma and Xiang Tang,
\emph{An Algebraic Index Theorem for Orbifolds},
Adv. Math., 210, 83-121, 2007.

\bibitem{Post}
Markus Pflaum, Hessel Posthuma and Xiang Tang, \emph{On the Algebraic Index for Riemannian \'Etale Groupoids}, Lett. Math. Phys., 90, 287-310, 2009.

\bibitem{SS1} Anton Savin, Elmar Schrohe, Boris Sternin, Uniformization and index of
    elliptic operators associated with diffeomorphisms of a manifold.
    Russ. J. Math. Phys. 22 (2015), no. 3, 410--420.

\bibitem{SS2} Anton Savin, Elmar Schrohe, Boris Sternin, On the index formula for
    an isometric diffeomorphism. (Russian) Sovrem. Mat. Fundam. Napravl.
    *46 * (2012), 141--152; translation in J. Math. Sci. (N.Y.)
    201:818–829 (2014)

\bibitem{SS3} Anton Savin, Elmar Schrohe, Boris Sternin, The index problem for
    elliptic operators associated with a diffeomorphism of a manifold
    and uniformization. (Russian) Dokl. Akad. Nauk 441 (2011), no. 5,
    593--596; translation in Dokl. Math. 84 (2011), no. 3, 846--849

\bibitem{Sternin}
Anton Savin, Boris Sternin
\emph{Elliptic theory for operators associated with diffeomorphisms of smooth manifolds
}, arXiv:1207.3017

\bibitem{steenrod}
Norman Steenrod,
\emph{The Topology of Fiber Bundles},
Princeton University Press, Princeton, 1951.


\end{thebibliography}
\end{document}